\documentclass[11pt,twoside]{amsart}
\usepackage{amssymb,latexsym,amsthm,amsmath,mathtools,amsfonts,hyperref,fancyhdr,comment,enumitem, cleveref}

\hypersetup{
	colorlinks=true,
	linkcolor=blue,
	filecolor=magenta,
	urlcolor=cyan,
	citecolor=red,
}

\long\def\symbolfootnote[#1]#2{\begingroup%
	\def\thefootnote{\fnsymbol{footnote}}\footnote[#1]{#2}\endgroup}

\newcommand{\lt}{\mathfrak l}

\newcommand{\p}{\mathcal P}

\newcommand{\Mod}{\mathrm{mod}}
\newcommand{\supp}{\mathrm{Supp}}

\makeatletter
\def\imod#1{\allowbreak\mkern10mu({\operator@font mod}\,\,#1)}
\makeatother

\newtheorem{theorem}{Theorem}[section]
\newtheorem{lemma}[theorem]{Lemma}
\newtheorem{corollary}[theorem]{Corollary}
\newtheorem{proposition}[theorem]{Proposition}
\newtheorem*{theorem*}{Theorem}
\newtheorem{definition}[theorem]{Definition}
\newtheorem{observation}[theorem]{Observation}
\newtheorem{remark}[theorem]{Remark}

\newtheorem{conjecture}[theorem]{Conjecture}
\numberwithin{equation}{section}
\newcommand{\ignore}[1]{}

\newcommand{\mynote}[1]{}
\newenvironment{manualtheorem}[1]{%
	\renewcommand\thetheorem{#1}%
	\theorem
}{\endtheorem}

\begin{document}

	\title{Alternating groups as products of cycle classes - II}
	
	\author{Harish Kishnani}
	\address{Indian Institute of Science Education and Research Mohali, Knowledge City, Sector 81, Mohali 140306, India}
	\email{harishkishnani11@gmail.com}
	
	\author{Rijubrata Kundu}
	\address{Indian Institute of Science Education and Research Mohali, Knowledge City, Sector 81, Mohali 140306, India}
	\email{rijubrata8@gmail.com}
	
	\author{Sumit Chandra Mishra}
	\address{Indian Institute of Science Education and Research Mohali, Knowledge City, Sector 81, Mohali 140306, India}
	\email{sumitcmishra@gmail.com}

	\today 
	\begin{abstract}
		Given integers $k,l\geq 2$, where either $l$ is odd or $k$ is even, let $n(k,l)$ denote the largest integer $n$ such that each element of $A_n$ is a product of $k$ many $l$-cycles. In 2008, M. Herzog, G. Kaplan and A. Lev  conjectured that $\lfloor \frac{2kl}{3} \rfloor \leq n(k,l)\leq \lfloor \frac{2kl}{3}\rfloor+1$. It is known that the conjecture holds when $k=2,3,4$. Moreover, it is also true when $3\mid l$. In this article, we determine the exact value of $n(k,l)$ when $3\nmid l$ and $k\geq 5$. As an immediate consequence, we get that $n(k,l)<\lfloor \frac{2kl}{3}\rfloor$ when $k\geq 5$, which shows that the above conjecture is not true in general. In fact, the difference between the exact value of $n(k,l)$ and the conjectured value grows linearly in terms of $k$. Our results also generalize the case of $k=2,3,4$.
	\end{abstract}
	
	\maketitle
	
	\section{Introduction}
	
	The study of products of conjugacy classes is a well established theme in the theory of finite groups. It also plays an important role in the theory of finite non-abelian simple groups. Many fundamental results as well as open problems exist in this direction. Let $G$ be a finite group and $C_1,C_2,\ldots, C_k$ be conjugacy classes of $G$. The product of these conjugacy classes, defined by $C_1C_2\cdots C_k:=\{x_1x_2\cdots x_r\mid x_i\in C_i, 1\leq i\leq r\}$ is invariant under conjugation and thus a union of conjugacy classes of $G$. In particular, for an integer $k\geq 2$, the $k$-th power of a conjugacy class, denoted by $C^k$ is simply $\underbrace{CC\cdots C}_{k\; times}$. A well-known conjecture of Arad and Herzog states that if $G$ is a finite non-abelian simple group, then the product of two non-trivial conjugacy classes is never a conjugacy class. This has been proved in the affirmative in many cases (see \cite{bfm}, Section 2 for a survey). On the other hand, the well-known Thompson's conjecture states that in a finite non-abelian simple group $G$, there exists a conjugacy class $C$ such that $C^2=G$. Hs\"u in \cite{xu} proved this for the alternating groups. Ellers and Gordeev (see \cite{eg}) achieved a break-through in this direction when they proved it for simple finite groups of Lie type over  field size greater than 8. Nonetheless, the conjecture remains open for these groups over small field size. Another fundamental result due to Guralnick and Malle (see \cite{gm}) states that there exists two conjugacy classes $C_1,C_2$ in a finite non-abelian simple group $G$ such that $C_1C_2\supseteq G\setminus \{1\}$.
	
	For a finite non-abelian simple group $G$, it is not difficult to see that there exists a natural number $m$ such that $C^m=G$ for every non-trivial conjugacy class $C$ of  $G$. The least natural number $m$ such that $C^m=G$ for every non-trivial conjugacy class $C$ of $G$  is called the covering number of $G$, denoted by $cn(G)$. More generally, the extended covering number denoted by $ecn(G)$ is the least number $m$ such that any collection of $m$ conjugacy classes $C_1,C_2,\ldots C_m$ satisfies $C_1\cdots C_m=G$. It was proved by Stavi that $cn(A_n)=\lceil\frac{n}{2}\rceil$. Dvir (see Chapter 3, \cite{ash}) developed many properties of products of conjugacy classes in $A_n$. He proved that $ecn(A_n)=cn(A_n)+1=\lceil \frac{n}{2} \rceil +1$. We refer the reader to \cite{ash}, for more on the covering number of finite simple groups.
	
	\medskip
	
	In this article, we continue the study of powers of cycle-classes in symmetric and alternating groups which has considerable literature. For an integer $2\leq l\leq n$, let $C_l$ denote the conjugacy class of $l$-cycles in $S_n$. Given natural numbers $k,l\geq 2$ where either $k$ is even or $l$ is odd, let $n(k,l)$ denote the largest natural number $n$ such that $C_l^k=A_n$, that is, every element in $A_n$ can be written as a product of $k$ many $l$-cycles.
	
	\medskip
	
	E. Bertram proved that $n(2,l)=\lfloor \frac{4l}{3} \rfloor+1$ when $l>2$ (see Theorem~\ref{n(2,l)_bounds}).  Bertram and Herzog obtained the exact value of $n(3,l)$ and $n(4,l)$ when $l>2$ (see Theorem~\ref{n(3,l)_bounds} and Theorem~\ref{n(4,l)_bounds}). In an attempt to generalize these results, the authors in \cite{hgl} gave a general upper bound for $n(k,l)$ (see Theorem~\ref{general_upper_bound_n(k,l)}). Further, they conjectured the following:
	
	\begin{conjecture}[Conjecture 1.1, \cite{hgl}]\label{mainconjecture}
		Let $k\geq 2$, $l\geq 3$ and assume that $l$ is odd or $k$ is even. Then $\lfloor \frac{2}{3}kl \rfloor\leq n(k,l)\leq \lfloor\frac{2}{3}kl\rfloor+1$.
	\end{conjecture}
	
	\noindent They verified the above conjecture when $3\mid l$, $l\geq 9$ and $l$ is odd. With these assumptions, they also proved that $n(k,l)=\frac{2}{3}kl$ when $k$ is odd (see Theorem~\ref{exact_value_for_n(k,l)_l_divisible_by_3_l_odd}). Very recently (see \cite{kkm}), we extended the above result with the same hypothesis except we assumed $k$ even (thus $l$ can be both odd or even) in which case it was seen that $n(k,l)=\frac{2}{3}kl+1$ (see Theorem~\ref{n(k,l)_exact_value_k_even_l_divisible_by_3}). This verified Conjecture 1.2 in \cite{hgl}. Further we also determined the value of $n(k,3)$ (see Theorem~\ref{n(k,3)_exact_value}). In this article, we prove that Conjecture~\ref{mainconjecture} does not hold. Furthermore, we determine the value of $n(k,l)$ when $k\geq 5$, $l\geq 2$ and $3\nmid l$. Our main theorem is as follows:
	
	\begin{manualtheorem}{A}\label{maintheorem}
		Let $k\geq 2$ and $l>3$ be two natural numbers such that $k$ is even, or $l$ is odd and $3\nmid l$. Then 
		\[n(k,l)=
			\begin{cases}
			\frac{2k(l-1)}{3}+2 & \text{when $l\equiv 1\;(\Mod\;3)$};\\
			\frac{k(2l-1)}{3}+1 & \text{when $l\equiv 2\;(\Mod\; 3)$, $k\not\equiv 1\;(\Mod\;4)$};\\
			\frac{k(2l-1)}{3}   & \text{when $l\equiv 2\;(\Mod\;3)$, $k\equiv 1\;(\Mod\;4)$}.
			\end{cases}
		\]
	\end{manualtheorem}
	
	\noindent We prove this theorem in three separate sections. In \Cref{Proof_of_main_theorem_l=1(mod 3)} we prove it for $l\equiv 1\;(\Mod\; 3)$. In \Cref{Proof_of_main_theorem_k_even_l=2(mod 3)} we prove it when $l\equiv 2\;(\Mod\; 3)$ and $k$ is even. Finally in \Cref{complete_proof}, we prove it when $l\equiv 2\;(\Mod\; 3)$ and $k$ is odd.
	\noindent We also have the following.
	\begin{manualtheorem}{B}\label{maintheorem2}
		For $k\geq 2$ even, we have $n(k,2)=k+2$.
	\end{manualtheorem}
	
	\noindent In \Cref{survey}, we prove the above theorem and develop some tools which will be required in the proof of our main theorem.

	The following two remarks are in order.
	
	\begin{remark}
		Note that $n(k,l)=\frac{2kl}{3}-\frac{k}{3}+\epsilon$ where $\epsilon\in \{0,1\}$ when $l\equiv 2\;(\Mod\; 3)$. Also $n(k,l)=\frac{2kl}{3}-\frac{2k}{3}+1$ when $l\equiv 1\;(\Mod\;3)$. Comparing with Conjecture~\ref{mainconjecture}, we see that the difference between the conjectured value and the actual value grows linearly in terms of $k$. Further, our expressions for the actual value also holds for $k=2,3,4$ when $3\nmid l$.
	\end{remark}
	
	\begin{remark}
		It is clear that $n(k,l)\geq l$. Note that from \Cref{maintheorem}, we have $n(k,l)\geq l+2$. Thus if $l$ is odd and $n'$ is such that $l+2\leq n'\leq n(k,l)$, then $C_l$ is a conjugacy class of $A_n$ and $C_l^k=A_{n'}$.
	\end{remark}
	 
	 Finally we mention that the results on products of conjugacy classes has considerable application in the theory of word maps on finite simple groups. For example, Bertram's result on $n(2,l)$ can be used to show that the word $x^Ny^N$, where $N=p^aq^b$ such that $p,q$ are primes, is surjective on the alternating groups. Moreover, it is also known to be surjective on all finite non-abelian simple groups (see \cite{GLOST}). On the other hand, the Thompson's conjecture mentioned above implies the Ore's Conjecture, which states that every element in a finite non-abelian simple group is a commutator. Ore's conjecture is now proved in the affirmative (see the beautiful article by Malle, \cite{ma} and references therein). We also refer the readers to an excellent survey article by Shalev (see \cite{sh}, Section 6) for more on this.
	 
	\subsection{Notations} We use some standard notations throughout this paper. Let $\sigma \in S_n$. If $\sigma$ is a cycle then $\mathfrak{l}(\sigma)$ denotes the length of $\sigma$. Let $\supp(\sigma):=\{i \mid \sigma(i)\neq i\}$ denote the support of $\sigma$. Two permutations $\sigma,\tau \in S_n$ are called disjoint if $\supp(\sigma)\cap \supp(\tau) = \emptyset$. It is well known that any permutation can be written as a product of disjoint cycles. Moreover, such a decomposition is unique up to cyclic shifts within each cycle and the order in which the cycles are written. For $\sigma \in S_n\setminus\{1\}$, $dcd*(\sigma)$ denotes  a disjoint cycle decomposition of $\sigma$ where each cycle has length greater than 1. Let $m_{\sigma}$ denote the number of symbols moved by $\sigma$, that is, $m_{\sigma}=|\supp(\sigma)|$. Let $n_{\sigma}$ denote the number of cycles in $dcd*(\sigma)$. For any $i\geq 2$, let $n_i(\sigma)$ denote the number of $i$-cycles in $dcd*(\sigma)$. It is clearly seen that $\sum_{i}n_i=n_{\sigma}$ and $\sum_{i}in_i=m_{\sigma}$. Finally, we mention that product of permutations will be executed from right to left.
	
	
	\section{A survey of known results and some generalities}\label{survey}
	
	In this section, we first collect some known results in detail which have been mentioned briefly in the introduction. Many of these results are required as they serve as base cases for  inductive proofs throughout this article. We also state and prove a general result (see \Cref{indecomposable}) which enable us to give counterexamples to Conjecture~\ref{mainconjecture}. In the sections to follow, we will use it to give improved upper bounds for $n(k,l)$ when $k\geq 5$ and $3\nmid l$. Finally we will record some more general observations and outline a rough strategy that we incorporate to prove lower bounds for $n(k,l)$ eventually leading to the proof of our main theorem.
	
	E. Bertram in \cite{be} determined the value of $n(k,l)$ for $k=2$. He proved the following result.
	
	\begin{theorem}[Corollary 2.1, \cite{be}]\label{n(2,l)_bounds} 
		Given natural numbers $n,l\geq 2$,  every element of $A_n$ can be written as a product of two $l$-cycles in $S_n$ if and only if $\lfloor\frac{3n}{4}\rfloor\leq l\leq n$ or $n=4$ and $l=2$. In other words, $n(2,l)=\lfloor\frac{4}{3}l\rfloor+1$ when $l>2$ and $n(2,2)=4$.
	\end{theorem}
 
	\noindent In \cite{bh}, the authors computed $n(3,l)$ and $n(4,l)$. 
	
	\begin{theorem}[Theorem 2, \cite{bh}]\label{n(3,l)_bounds}
		For $n,l\geq 2$, each element $\sigma\in A_n$ is a product of three $l$-cycles in $S_n$ if and only if $l$ is odd and either $\lceil\frac{n}{2}\rceil\leq l \leq n$, or $n=7$ and $l=3$. In other words, for $l$ odd and $l\neq 3$, $n(3,l)=2l$  and $n(3,3)=7$.
	\end{theorem}

	\begin{theorem}[Theorem 3, \cite{bh}]\label{n(4,l)_bounds}
		For $n,l\geq 2$, each element $\sigma\in A_n$ is a product of four $l$-cycles in $S_n$ if and only if:
		\begin{enumerate}
			\item $\lceil \frac{3n}{8} \rceil \leq l\leq n$ if $n\not\equiv 1\;(\Mod\;8)$,
			\item $\lfloor\frac{3n}{8}\rfloor \leq l\leq n$ if $n\equiv 1\;(\Mod\;8)$,
			\item $n=6$ and $l=2$.
		\end{enumerate}
		In other words, $n(4,l)=\lfloor \frac{8l}{3} \rfloor +1$ when $3\mid l$ or $l=2$, and $n(4,l)=\lfloor \frac{8l}{3}\rfloor$ when $l>3$ and $3\nmid l$. 
	\end{theorem}
	
	\begin{remark}\label{newvalues_k=2_k=3_k=4}
		From the above three theorems, we observe that when $l>3$,
		\begin{enumerate}
			\item \[ n(2,l)=\left\lfloor \frac{4l}{3} \right\rfloor +1=
			\begin{cases}
			\frac{4l+2}{3} & \text{when $l\equiv 1\;(\Mod\;3)$}\\
			\frac{4l+1}{3} & \text{when $l\equiv 2\;(\Mod\;3)$};
			\end{cases}
			\]
			
			\medskip
			
			\item \hspace{1.8 cm} $n(3,l)=2l$ where $l$ is odd;
			
			\medskip
			
			\item \[\hspace{-5 mm}n(4,l)=\left\lfloor \frac{8l}{3} \right\rfloor =
			\begin{cases}
			\frac{8l-2}{3} & \text{when $l\equiv 1\;(\Mod\;3)$}\\
			\frac{8l-1}{3} & \text{when $l\equiv 2\;(\Mod\;3)$}.
			\end{cases}
			\]
			
		\end{enumerate}
		Thus the expressions in the above theorems matches the corresponding expressions in \Cref{maintheorem}. From now on, for $l>3$ and $3\nmid l$, we will use these new expressions for $n(k,l)$ when $k=2,3,4$.
	\end{remark}
	Some general results were obtained by Herzog, Lev and Kaplan in \cite{hgl}.  First they gave upper bounds for $n(k,l)$.
	
	\begin{theorem}[Theorem 3.3, \cite{hgl}]\label{general_upper_bound_n(k,l)}
		Let $k,l$ be natural numbers such that $k\geq 2$ and $l>2$. Suppose that either $l$ is odd or $k$ is even. Denote $n_1=\lfloor\frac{2kl}{3}\rfloor$ and $\delta=\frac{2kl}{3}-n_1$. Then
		\begin{enumerate}
			\item If $n_1\equiv 0\;(\rm{mod}\;4)$, then $n(k,l)\leq n_1+1$.
			\item If $n_1\equiv 1\;(\rm{mod}\;4)$, then $n(k,l)\leq n_1+1$.
			\item If $n_1\equiv 2\;(\rm{mod}\;4)$, then $n(k,l)\leq n_1+1$; if we further assume that $l>3$ and $\delta\in \{0,\frac{1}{3}\}$, then $n(k,l)\leq n_1$;
			\item If $n_1\equiv 3\;(\rm{mod}\;4)$, then $n(k,l)\leq n_1$.
		\end{enumerate}
	\end{theorem}
	
	\noindent From the above result it is clearly seen that $n(k,l)\leq \lfloor \frac{2kl}{3} \rfloor+1$ when $l>2$. They further gave a lower bound assuming some conditions on $l$ and $k$.
	
	\begin{theorem}[Theorem 3.4, \cite{hgl}]\label{exact_value_for_n(k,l)_l_divisible_by_3_l_odd}
		Let $k\geq 2$ and $3\mid l$, $l\geq 9$ and $l$ is odd. Then $\frac{2}{3}kl\leq n(k,l)\leq \frac{2}{3}kl+1$. Furthermore, if $k$ is odd then $n(k,l)=\frac{2}{3}kl$.
	\end{theorem}
	
	In the view of the above two theorems they proposed Conjecture~\ref{mainconjecture} (see introduction). Furthermore, they conjectured that $n(k,l)=\frac{2}{3}kl+1$ when $k$ is even and $3\mid l$ (see Conjecture 1.2, \cite{hgl}). We proved it in the affirmative in \cite{kkm}. The main results of \cite{kkm} are as follows:
	
	\begin{theorem}[Theorem A, \cite{kkm}]\label{n(k,3)_exact_value}
		Suppose $k\geq 2$ be any natural number. Then $n(k,3)=2k+1$.
	\end{theorem}

	\begin{theorem}[Theorem B, \cite{kkm}]\label{n(k,l)_exact_value_k_even_l_divisible_by_3}
		Let $k\geq 2$ and $l>3$ be two natural numbers such that $k$ is even and $3\mid l$. Then $n(k,l)=\frac{2}{3}kl+1$.
	\end{theorem}
	
	Now we proceed to state and prove one of the main result of this section (\Cref{indecomposable}), which will enable us to give improved upper bounds for $n(k,l)$ when $k\geq 5$ and $3\nmid l$, thereby answering Conjecture~\ref{mainconjecture} in the negative. We start with an easy observation about even permutations.
	
	\begin{observation}\label{even_permutation_characterization}
		Let $\sigma\in S_n$. Then $\sigma \in A_n$ if and only if $m_{\sigma}+n_{\sigma}$ is even.
	\end{observation}
	
	\medskip
	
	\noindent \textbf{Notation:} Let $\p(k,l;n)$ denote the set of all $\sigma \in S_n$ such that $\sigma$ can be written as a product of $k$ many $l$-cycles in $S_n$.
	
	\medskip
	
	Note that if $k$ is even, or $l$ is odd then $\p(k,l;n)\subseteq A_n$. Otherwise, $\p(k,l;n)\subseteq S_n\setminus A_n$. The following is an easy observation.
	
	\begin{observation}\label{increasing_cycles}
		Let $k,l\in \mathbb{N}$ and $l\geq 2$. If $l$ is odd then $\p(k,l;n)\subseteq \p(k+1,l;n)$. If $l$ is even then $\p(k,l;n)\subseteq \p(k+2,l;n)$.
	\end{observation}
	
	\noindent We will use \Cref{even_permutation_characterization} and \Cref{increasing_cycles} throughout the article without any further mention.
	
	Let $\sigma \in S_n$. Then if $\sigma \in \p(k,l;n)$, $\sigma$ must satisfy the following necessary condition.  
	
	\begin{lemma}[Corollary 3.2, \cite{hgl}]\label{general_product_condition}
		Let $\sigma \in S_n$ and $\sigma=C_1C_2\cdots C_k$ where $C_i$ is a cycle of length $l$ for every $1\leq i\leq k$. Suppose that there does not exist a cycle of length $l$ in $dcd*(\sigma)$. Then $m_{\sigma}+n_{\sigma}\leq kl$.
	\end{lemma}
	 The following is a well-known theorem of Ree (see \cite{re}). A combinatorial proof of the result can be found in \cite{fls} (see also Theorem 4, \cite{ash}). The statement that follows is taken from \cite{bh}.
	
	\begin{theorem}[Theorem 4, \cite{bh}]\label{theorem_of_Ree}
		Let $D_i\in S_n$ for $1\leq i\leq k$ and $D_1D_2\cdots D_k=1$. Suppose $H=\langle D_1,\ldots, D_k\rangle$. If $T$ denotes the number of orbits of the natural action of $H$ on $\{1,2,\ldots,n\}$, then 
		$$\sum_{i=1}^{k}(m_{D_i}-n_{D_i})\geq 2(n-T).$$
	\end{theorem}

	\noindent As a corollary of the above the result, we get the following:
	\begin{corollary}\label{ree_corollary}
		Suppose $\sigma \in S_n$ and $\sigma=C_1C_2\cdots C_k$ where $C_i$ is a cycle of length $l$ for every $1\leq i\leq k$ and $\displaystyle \cup_{i=1}^{k} \supp(C_i)=\supp(\sigma)$. If $T$ is the number of orbits of the natural action of $H=\langle C_1,\ldots, C_k\rangle$ on $\{1,2,\ldots, m_{\sigma}\}$ then
		$$kl-m_{\sigma}-n_{\sigma}\geq k-2T.$$
	\end{corollary}

	\begin{proof}
		The proof directly follows from \Cref{theorem_of_Ree}.
	\end{proof}
	
	We now define the notion of an indecomposable permutation in our context.
	
	\begin{definition}
		Let $k,l\geq 2$ and $l\leq n$. Suppose that $\sigma \in S_n$ be such that $\sigma\in \p(k,l;n)$. Let $\sigma=\sigma_1\sigma_2\cdots \sigma_r$ be the disjoint cycle decomposition of $\sigma$ where $\lt(\sigma_j)\geq 2$ and $\lt(\sigma_j)\neq l$ for all $j$. For $2\leq k'\leq k-2$, we say $\sigma$ is $(k',k-k')$ decomposable if $\sigma=\rho \tau$ such that $\supp(\rho)\cap \supp(\tau)=\emptyset$ and $dcd*(\sigma)=dcd*(\rho)\cup dcd*(\tau)$, and $\rho=C_1\cdots C_{k'}$, $\tau=C_{k'+1}\cdots C_{k}$ where $C_i$'s are $l$-cycles. We say that $\sigma$ is \textbf{k-decomposable} if $\sigma$ is $(k',k-k')$ decomposable for some $k'$ where $2\leq k'\leq k-2$. Further, when $k$ is even, we say that $\sigma$ is \textbf{even-k-decomposable} if $\sigma$ is $(k',k-k')$ decomposable for some even  $k'$ such that $2\leq k'\leq k-2$. We say that $\sigma$ is \textbf{k-indecomposable} (resp. \textbf{even-k-indecomposable}) if $\sigma$ is not $k$-decomposable (resp. even-$k$-decomposable).
	\end{definition}

	\begin{lemma}\label{indecomposable}
		Let $\sigma \in S_n$ and $\sigma=\sigma_1\cdots \sigma_r$ be the disjoint cycle decomposition of $\sigma$ (excluding the trivial 1-cycles). Suppose $\lt(\sigma_j)\neq l$ for every  $1\leq j \leq r$. Suppose that $k\geq 4$ and $\sigma\in \p(k,l;n)$. Then
		\[kl-m_{\sigma}-n_{\sigma}\geq 
			\begin{cases}
				k-2 & \text{when $\sigma$ is $k$-indecomposable}\\
				k-4 & \text{when $k$ is even and  $\sigma$ is even-$k$-indecomposable.}
			\end{cases}
		\]
	\end{lemma}

	\begin{proof}
		Without loss of generality we can assume $n=m_{\sigma}$. Let $\sigma=C_1\cdots C_k$ where each $C_i$ is a $l$-cycle. Let $H=\langle C_1,\ldots,C_k\rangle$ and consider the natural action of $H$ on $\{1,2,\ldots, m_{\sigma}\}$. Let $T$ denote the number of orbits of this action. Let us consider a cycle $C_{i_1}$ where $1\leq i_1\leq k$. Let $a\in \supp(C_{i_1})$. Let $[a]$ denote the orbit of $a$. It is clear that $\supp(C_{i_1})\subset [a]$. Since $\lt(\sigma_j)\neq l$ for every $1\leq j\leq r$, we conclude that there exists $C_{i_2}$ for some $i_2\neq i_1$ such that $\supp(C_{i_1})\cap \supp(C_{i_2})\neq \emptyset$. This immediately implies that $\supp(C_{i_2})\subset [a]$. Suppose that  for every $1\leq i_3\leq k$ such  that $i_3\neq i_1,i_2$ we have $(\supp(C_{i_1})\cup \supp(C_{i_2}))\cap \supp(C_{i_3})=\emptyset$. It follows that $\sigma$ is $(2,k-2)$ decomposable, contrary to our assumption in both cases. Thus there exists $i_3\neq i_1,i_2$ such that $(\supp(C_{i_1})\cup \supp(C_{i_2}))\cap \supp(C_{i_3})\neq \emptyset$, whence it follows that $\supp(C_{i_3})\subset [a]$.
		
		\medskip
		
		\textbf{\underline{Case I}}: Suppose now that $\sigma$ is $k$-indecomposable. Proceeding with the above argument we get that $\cup_{j=1}^{k-1} \supp(C_{i_j})\subset [a]$ for some subset $\{i_1,i_2,\ldots,i_{k-1}\}\subset \{1,\ldots,k\}$. Finally we are left with a single cycle $C_{i_k}$ and since $\lt(\sigma_i)\neq l$ for every $1\leq i\leq r$ we conclude that $\cup_{j=1}^{k-1} \supp(C_{i_j})\cap C_{i_k}\neq \emptyset$. This implies that $\supp(C_{i_k})\subset [a]$, thereby proving that the action is transitive. Thus $T=1$. Applying this to Corollary~\ref{ree_corollary}, we get our desired result.
		
		\medskip
		
		\textbf{\underline{Case-II}:} Suppose now that $k$ is even and $\sigma$ is even-$k$-indecomposable. In this case we argue that $T\leq 2$, from which the result will follow once again applying Corollary~\ref{ree_corollary}. Recall that $\supp(C_{i_1})\cup \supp(C_{i_2})\cup \supp(C_{i_3})\subset [a]$. Let us assume that $T\neq 1$. Thus there exists $C_{i_4}$ such that $[\cup_{j=1}^{3} \supp(C_{i_j})] \cap \supp(C_{i_4})=\emptyset$. Let $b\in \supp(C_{i_4})$. Then it readily follows that $[b]\cap [a]=\emptyset$ and $\supp(C_{i_4})\subset [b].$ Note that since $\lt(\sigma_j)\neq l$ for every $j$, there must exist $C_{i_5}$ such that $\supp(C_{i_4})\cap \supp(C_{i_5})\neq \emptyset$, whence it follows that $\supp(C_{i_5})\subset [b]$. Now since $\sigma$ is even-$k$-indecomposable,  there exists $C_{i_6}$ such that $(\supp(C_{i_4})\cup \supp(C_{i_5}))\cap \supp(C_{i_6})\neq \emptyset$ and hence $\supp(C_{i_6})\subset [b]$. If $k=6$, then $T=2$ and we are done. Let us thus assume that $k>6$. Observe that the number of cycles among $C_1,\ldots, C_k$ whose support constitute a single orbit must be odd, or else our assumption that $\sigma$ is even-$k$-indecomposable  fails to hold. Thus if there are more than two orbits then we can take the cycles which constitute any two of these orbits (the number of combined cycle constituents, say $k'$, will then be even) and their product taken together (in some order) yields a product $\sigma_{t_1}\cdots \sigma_{t_f}$ where $\{t_1,\ldots,t_f\}\subsetneq \{1,2,\ldots,r\}$. This forces $\sigma$  to be $(k',k-k')$ decomposable, contrary to our assumption. Thus we conclude that $T=2$. Overall, in this case, we get that $T\leq 2$ and our proof is complete.
	\end{proof}

	As a consequence we have the following lemma which will be needed later.
	
	\begin{lemma}\label{adding_even_many_two_cycles}
		Let $\sigma\in S_n$ be such that $kl-m_{\sigma}-n_{\sigma}<k-2$ and 
		$\sigma\notin \p(k,l;n)$. Let $(\alpha \beta)(\gamma \delta)\in A_n$ be such that $\alpha, \beta, \gamma, \delta$ are distinct and $\{\alpha,\beta,\gamma,\delta\}\cap \supp(\sigma)=\emptyset$. 
		Suppose that one of the following holds
		\begin{enumerate}
			\item $dcd*(\sigma)$ consists of $2$-cycles only,
			\item $dcd*(\sigma)$ consists of $2$-cycles only except for one $4$-cycle,
			\item $dcd*(\sigma)$ consists of $2$-cycles only except for one $3$-cycle.
		\end{enumerate}
		Then $\sigma (\alpha \beta)(\gamma \delta) \notin \p(k,l;n)$.
	\end{lemma}
	
	\begin{proof}
		Let $\sigma'=\sigma (\alpha \beta)(\gamma \delta)$. We have $m_{\sigma'}=m_{\sigma}+4$ and $n_{\sigma'}=n_{\sigma}+2$. Thus, by our assumption on $\sigma$, $kl-m_{\sigma'}-n_{\sigma'}<(k-2)-4-2=k-8$. It is enough to show that $\sigma'$ is $k$-indecomposable by \Cref{indecomposable}. Assume on the contrary that $\sigma'$ is $k$-decomposable. Thus $\sigma'=\rho \tau$ where $\rho, \tau$ are disjoint permutations in $S_n$ and there exists $k'$, $2\leq k' \leq k-2$ such that $\rho \in \p(k',l;n)$ and $\tau\in \p(k-k',l;n)$. We can write $\rho=C_{i_1}C_{i_2}\cdots C_{i_{k'}}$ and $\tau=C_{i_{k'+1}}\dots C_{i_{k}}$.
		
		\medskip
		\textbf{\underline{Case 1}:} Suppose that $\{\alpha,\beta\} \cap \supp(\rho) \neq \emptyset$ and $\{\gamma,\delta\}\cap \supp(\tau) \neq \emptyset$. Replace the symbol $\alpha$ in $C_{i_1}\cdots C_{i_{k'}}$ by the symbol $\gamma$ and the symbol $\beta$ by $\delta$ to get $C'_{i_1}\cdots C'_{i_{k'}}$. Then $C'_{i_t}$ is again an $l$-cycle for all $t$, $1 \leq t \leq k'$ and $C'_{i_1}\cdots C'_{i_{k'}}=\rho(\alpha \beta)(\gamma \delta)$. Consequently, $\sigma = \rho (\alpha \beta) (\gamma \delta) \tau = C'_{i_1}\dots C'_{i_{k'}}C_{i_{k'+1}}\cdots C_{i_{k}}$. Thus $\sigma$ is $k$-decomposable, which is a contradiction as $\sigma \notin \p(k,l;n)$. 
		
		\medskip
		
		\textbf{\underline{Case 2}:} Suppose that either $\{\alpha,\beta,\gamma,\delta\}\cap \supp(\rho)=\emptyset$ or $\{\alpha,\beta,\gamma,\delta\}\cap \supp(\tau)=\emptyset$. WLOG we may assume that $\{\alpha,\beta,\gamma,\delta\}\cap \supp(\rho)=\emptyset$. 
		
		\medskip
		
		\noindent \textbf{\underline{Subcase 2A}:} Suppose that there exists a 2-cycle in $dcd*(\rho)$. We fix a 2-cycle $(a,b)\in dcd*(\rho)$ and replace  the symbols $a$ and $b$ in $C_{i_1}\cdots C_{i_{k'}}$ by $\alpha$ and $\beta$ respectively to yield a product $C'_{i_1}\cdots C'_{i_{k'}}$. Then $C'_{i_t}$ is again an $l$-cycle for all $t$, $1 \leq t \leq k'$ and $\rho(a b)(\alpha \beta)=C'_{i_1}\cdots C'_{i_{k'}}$. Similarly, replace the symbols $\gamma$ and $\delta$ in $C_{i_{k'+1}}\cdots C_{i_k}$ by $a$ and $b$ respectively to yield a product $C'_{i_{k'+1}}\cdots C'_{i_k}$. Once again $C'_{t}$ is a $l$-cycle for all $t$, $k'+1 \leq t \leq k$ and $\tau(\gamma \delta)(a b)=C'_{i_{k'+1}}\cdots C'_{i_k}$. Consequently, $\sigma = \rho (a b)(\alpha \beta)\tau (\gamma \delta)(a b)=C'_{i_1}\cdots C'_{i_{k'}}C'_{i_{k'+1}}\cdots C'_{i_k}$. Thus $\sigma$ is $k$-decomposable, which is a contradiction as $\sigma \notin \p(k,l;n)$.
		
		\medskip
		
		\noindent \textbf{\underline{Subcase 2B}:} Suppose that $\rho$ is a $3$-cycle or a $4$-cycle.
		If $\rho$ is a $3$-cycle then $m_{\tau}=m_{\sigma'}-3=m_{\sigma}+4-3=m_{\sigma}+1$ and $n_{\tau}=n_{\sigma'}-1=n_{\sigma}+2-1=n_{\sigma}+1$, which implies that $(k-k')l-m_{\tau}-n_{\tau}=kl-m_{\sigma}-n_{\sigma}-2-k'l<k-k'-2$. 
		Similarly, if $\rho$ is a $4$-cycle then $m_{\tau}=m_{\sigma'}-4=m_{\sigma}+4-4=m_{\sigma}$ and $n_{\tau}=n_{\sigma'}-1=n_{\sigma}+2-1=n_{\sigma}+1$, which implies that $(k-k')l-m_{\tau}-n_{\tau}=kl-m_{\sigma}-n_{\sigma}-1-k'l<k-k'-2$. 
		Thus $(k-k')l-m_{\tau}-n_{\tau}<k-k'-2$. Since $\tau \in P(k-k',l;n)$, $\tau $ is ($k-k'$)-decomposable by \Cref{indecomposable}. Thus there exists disjoint permutations $\tau_1, \tau_2 \in S_n$ such that $\tau=\tau_1 \tau_2$ and 
		$\tau_1=C_{j_1} \cdots C_{j_{k''}}$, $\tau_2=C_{j_{k''+1}}\cdots C_{j_{k-k'}}$ 
		where $({j_1}, \dots ,j_{k''},j_{k''+1}, \dots ,j_{k-k'})=(i_{k'+1},\dots ,i_{k})$, 
		upto a permutation possibly. Thus $\sigma'=\rho\tau_1 \tau_2$. 
		Now if $\{\alpha, \beta\} \cap \supp(\tau_1)\neq \emptyset $ and $\{\gamma, \delta\}\cap \supp(\tau_2)\neq \emptyset$, or $\{\alpha,\beta\}\cap \supp(\tau_2)\neq \emptyset $ and $\{\gamma,\delta\}\cap \supp(\tau_1)\neq \emptyset$ then arguing as in Case 1, we get that $\sigma=(\rho\tau_1)\tau_2$ is $k$-indecomposble, which is a contradiction. If $\{\alpha,\beta,\gamma,\delta\}\cap \supp(\tau_1)=\emptyset$, then considering $\sigma' =(\rho \tau_1)\tau_2$ and arguing as in Subcase 2A, we get that $\sigma$ is $k$-decomposable, which is a contradiction. Similarly, if $\{\alpha ,\beta,\gamma,\delta\}\cap \supp(\tau_2)=\emptyset$ then considering $\sigma'=(\rho \tau_2)\tau_{1}$ and arguing as in Subcase 2A, we get that $\sigma$ is $k$-decomposable, which is a contradiction. 
	\end{proof}
	
	\noindent As an yet another consequence of \Cref{indecomposable}, we prove \Cref{maintheorem2}.
	\begin{proof}[\textbf{Proof of \Cref{maintheorem2}}]
		Consider $\sigma=(1\;2\;\cdots\;k+3)\in A_{k+3}$ since $k$ is even. Clearly $\sigma$ is $(k+3)$-indecomposable. Note that $kl-m_{\sigma}-n_{\sigma}=2k-k-3-1=k-4<k-2$. By \Cref{indecomposable}, we conclude that $\sigma \notin \p(k,l;k+3)$. Thus $n(k,2)\leq k+2$.
		
		To complete the proof we show that $n(k,l)\geq k+2$. Let $\sigma \in A_{k+2}$. 
		
		\noindent \textbf{\underline{Case I}:} Let $n_{\sigma}=1$. Let $\lt(\sigma)=t$ where $t$ is odd and $t\leq k+1$. Clearly $\sigma$ can be written as a product of $t-1$ many 2-cycles. Since $t-1\leq k$ and both $t-1$ and $k$ are even, we conclude that $\sigma \in \p(k,2; k+2)$.
		
		\noindent \textbf{\underline{Case II}:} Let $n_{\sigma}\geq 2$. Let $\sigma=\sigma_1\sigma_2\cdots \sigma_{n_{\sigma}}$ be the disjoint cycle decomposition of $\sigma$. Then $\sigma$ can be written as a product of $\sum_{i=1}^{n_{\sigma}} (m_{\sigma_i}-1)$ many 2-cycles. Since $\sum_{i=1}^{n_{\sigma}}( m_{\sigma_i}-1)=m_{\sigma}-n_{\sigma}$, we conclude that $\sigma \in \p(m_{\sigma}-n_{\sigma},l; k+2)$. Since $m_{\sigma}-n_{\sigma}$ is even and $m_{\sigma}-n_{\sigma}\leq k$, we conclude that $\sigma \in \p(k,l;k+2)$. Thus $n(k,2)\geq k+2$ and the proof is complete.
	\end{proof}
	
	Now we state and prove some more general results which will contribute towards the proof of our main theorem. The following theorem of \cite{hkl} will be required.
	
	\begin{theorem}[Theorem 7, \cite{hkl}]\label{product_of_two_cycles}
		Let $\sigma \in S_n$ and let $l_1,l_2\in \mathbb{N}$, $n\geq l_1\geq l_2\geq 2$. Then $\sigma=C_1C_2$, where $C_1,C_2$ are cycles in $S_n$ of lengths $l_1,l_2$ respectively, if and only if either $n_{\sigma}=2$, $l_1,l_2$ are the lengths of the cycles in $dcd*(\sigma)$ and $l_1+l_2=m_{\sigma}$, or the following conditions hold:
		
		\begin{enumerate}
			\item $l_1+l_2=m_{\sigma}+n_{\sigma}+2s$ for some $s\in \mathbb{Z}_{\geq 0}$, and
			\item $l_1-l_2\leq m_{\sigma}-n_{\sigma}.$
		\end{enumerate}
	\end{theorem}
	
	\noindent The following two lemmas will also be needed.
	
	\begin{lemma}[Lemma 3.6, \cite{kkm}]\label{cycles_as_products_of_cycles}
		Let $l,k\in \mathbb{N}$ such that $l\geq 2$ and $k\geq 1$. Suppose that $\sigma\in S_n$ be a cycle such that $\mathfrak{l}(\sigma)=l+(k-1)(l-1)$. Then $\sigma\in \p(k,l;n)$.
	\end{lemma}

	\begin{lemma}[Lemma 3.7, \cite{kkm}]\label{lengthening_of_cycles}
		Suppose $l,k\geq 2$ and $n\geq 4$ be natural numbers such that when $k$ is even, $2\leq l \leq n-1$ and when $k$ is odd, $2\leq l\leq n-2$. Then $\p(k,l;n)\subseteq \p(k,l+1;n)$ when $k$ is even, and $\p(k,l;n)\subseteq \p(k,l+2;n)$ when $k$ is odd.
	\end{lemma}
	
	\noindent The following two propositions will play a crucial role in the proof of the main theorem.
	\begin{proposition}\label{bounded_part_partitions_l=2(mod 3)}
		Let $k\geq 5$, $l\equiv 2\;(\Mod\;3)$ and $3<l\leq n \leq \frac{k(2l-1)}{3}+1$. Suppose $\sigma \in A_n$ be such that $n_{\sigma}\leq \frac{k}{3}(l-2)+1$. Then $\sigma \in \p(k,l;n)$.
	\end{proposition}

	\begin{proof}
		We divide the proof in two cases. At first we assume that $k$ is even and $k\geq 6$. In this case, take $l_1=l_2=l+(\frac{k}{2}-1)(l-1)$. Note that $l_1-l_2=0\leq m_{\sigma}-n_{\sigma}$. Further
		$$m_{\sigma}+n_{\sigma}\leq k(l-1)+2=l_1+l_2.$$ 
		Since $m_{\sigma}+n_{\sigma}$ and $l_1+l_2$ are both even, applying \Cref{product_of_two_cycles} we conclude that $\sigma=C_1C_2$ where $C_1$ and $C_2$ are cycles each of length $l+(\frac{k}{2}-1)(l-1)$. From \Cref{cycles_as_products_of_cycles}, both $C_1$ and $C_2$ can be written as a product of $k/2$ many $l$-cycles, whence it follows that $\sigma \in \p(k,l;n)$.
		
		Now assume that $k\geq 5$ and $k$ is odd. Note that $l$ must be odd. We make some cases.
		
		\textbf{\underline{Case-I}:} Suppose that $3\leq m_{\sigma}\leq l-1$. Either $m_{\sigma}$ or $m_{\sigma}-1$ is an odd integer which we call $l_1$. Then $\lceil \frac{m_{\sigma}}{2}\rceil \leq m_{\sigma}-1\leq l_1\leq m_{\sigma}$. Using Theorem~\ref{n(3,l)_bounds}, we conclude that $\sigma \in \p(3,l_1;n)$. Since $l_1$ is odd, we get that $\sigma \in \p(k,l_1;n)$. Finally as $l_1< l$ and both $l_1$ and $l$ are odd, using \Cref{lengthening_of_cycles}, we get $\sigma\in \p(k,l;n)$.
		
		\medskip
		
		\textbf{\underline{Case-II}:} Suppose that $l\leq m_{\sigma}\leq 2l-1$. Then,  $\lceil\frac{m_{\sigma}}{2}\rceil\leq l\leq m_{\sigma}$. From Theorem~\ref{n(3,l)_bounds} we conclude that $\sigma \in \p(3,l;n)$. Since $l$ is odd, $\sigma \in \p(k,l;n)$ as required.
		
		\medskip
		
		\textbf{\underline{Case-III}:}  Note that $m_{\sigma}\geq 2n_{\sigma}$. Thus, in this final case, we assume that $m_{\sigma}\geq \mathrm{max}\{2l,2n_{\sigma}\}$. Clearly,
		$$m_{\sigma}-n_{\sigma}\geq \mathrm{max}\{2l-n_{\sigma}, n_{\sigma}\}\geq l>l-1.$$
		Take $l_1=l+(\frac{k+1}{2}-1)(l-1)$ and $l_2=l+(\frac{k-1}{2}-1)(l-1)$. Note that $l_1\geq l_2$ and $l_1-l_2=l-1$. Thus $m_{\sigma}-n_{\sigma}\geq l_1-l_2$. Also $l_1+l_2=k(l-1)+2$. As in the previous case, we get that
		$$m_{\sigma}+n_{\sigma}\leq k(l-1)+2=l_1+l_2.$$
		Note further that $m_{\sigma}+n_{\sigma}$ is even and so is $l_1+l_2$ as $l$ is odd.
		Using \Cref{product_of_two_cycles} we conclude that $\sigma=C_1C_2$ where $C_1$ and $C_2$ are cycles of length $l_1$ and $l_2$ respectively. Using \Cref{cycles_as_products_of_cycles}, we get that $C_1$ is a product of $\frac{k+1}{2}$ many $l$-cycles and $C_2$ is a product of $\frac{k-1}{2}$ many $l$-cycles, whence it follows that $\sigma \in \p(k,l;n)$. This completes the proof.
	\end{proof}
	
	\begin{proposition}\label{bounded_part_partitions_l=1(mod 3)}
		Let $k\geq 5$, $l\equiv 1\;(\Mod\;3)$ and $3<l\leq n \leq \frac{2k(l-1)}{3}+2$. Suppose $\sigma \in A_n$ such that $n_{\sigma}\leq \frac{k}{3}(l-1)$. Then $\sigma \in \p(k,l;n)$.
	\end{proposition}
	
	\begin{proof}
		Using the bounds $m_{\sigma}\leq \frac{2k(l-1)}{3}+2$ and $n_{\sigma}\leq \frac{k(l-1)}{3}$ and arguing as in the previous proposition, the proof follows.
	\end{proof}
	
	To end this section, we outline a rough strategy behind the proof of \Cref{maintheorem}.
	
	\subsection{A rough strategy} For $k\geq 2$ and $l>3$ with $3\nmid l$, let $n=n(k,l)$ be as in Theorem~\ref{maintheorem}. In the previous two propositions we have proved that if  $\sigma\in A_n$ is such that $n_{\sigma}\leq f(k,l)$ (where $f(k,l)$ is determined depending on $l\equiv 1\;(\Mod\; 3)$ or $l\equiv 2\;(\Mod\; 3)$) then $\sigma \in \p(k,l;n)$. When $\sigma$ is such that $n_{\sigma}>f(k,l)$, to show that $\sigma \in \p(k,l;n)$, we simply show that such a $\sigma$ is $k$-decomposable. This is established using induction on $k$ (since by Remark~\ref{newvalues_k=2_k=3_k=4} we have the base cases). To facilitate the induction, we prove that $\sigma$ can be written as a product of two disjoint even permutations, say $\rho$ and $\tau$, such that $\supp(\rho)$ and $\supp(\tau)$ are suitably smaller than $n$ (see \Cref{general_decomposibility_l(2mod3)_k_even}, \Cref{decomposability_lemma_k=3(mod 4)_l_odd_2(mod3)}, \Cref{decomposability_lemma_k=1(mod 4)_l_odd_2(mod3)}).  To do such a decomposition, we  use a simple technique of approximation using small length cycles. Recall that for $\sigma \in S_n$, $n_i:=n_{i}(\sigma)$ denote the number of $i$-cycles in $\sigma$. We have the following inequality
	\begin{equation}\label{cycleinequality}
	2n_2+3n_3+\cdots +in_i + (i+1)(n_{\sigma}-n_2-n_3-\cdots-n_i)\leq m_{\sigma}.
	\end{equation}
	Suppose now that for $2\leq t\leq i$, $n_t\leq a_t$ except for some natural number $j$ in between $2$ and $i$. Here $i\geq 2$ and $a_t$'s are some natural numbers. Further, if $m_{\sigma}\leq n$ for some natural number $n$ and $n_{\sigma}\geq b$, then using \Cref{cycleinequality} we can obtain a least estimate for $n_j$ as follows:
	\begin{equation}\label{cycle_bound_inequality}
	(i-j+1)n_j\geq (i+1)n_{\sigma}-m_{\sigma}-\sum_{\substack{2\leq t\leq i \\ t\neq j}}(i+1-t)n_t\geq (i+1)b-n-\sum_{\substack{2\leq t\leq i \\ t\neq j}}(i+1-t)a_t.
	\end{equation}
	
	\noindent We will use this estimate above to conclude that there are enough small length cycles in $dcd*(\sigma)$ (since $n_{\sigma}$ is \emph{``large enough''}). We then use these cycles to construct a suitable decomposition of $\sigma$. In fact, cycles of length two and three will suffice. This establishes that the desired value of $n(k,l)$ is indeed a lower bound. 
	
	\medskip
	
	To show that it is also an upperbound, we construct suitable permutations in $A_{n+1}$ which do not belong to $\p(k,l;n+1)$ (see \Cref{UB_for_l=1(mod 3)}, \Cref{UB_for_odd_l}, \Cref{UB_for_l_odd_k_odd_l=2(mod3)}). In this case, we show that such a permutation is neither $k$-decomposable, nor it satisfies the inequalities in \Cref{indecomposable}, thereby showing that they cannot be written as a product of $k$ many $l$-cycles.

	
	\section{Exact value of $n(k,l)$ when $l\equiv 1\;(\Mod\;3)$}\label{Proof_of_main_theorem_l=1(mod 3)}
	
	In this section we prove \Cref{maintheorem} in the case when $l\equiv1\;(\Mod\;3)$. We start by proving that the desired value of $n(k,l)$ in this case is a lower-bound.
	
	\begin{proposition}\label{lowerbound_n(k,l)_l=1(mod 3)}
		Let $k\geq 5$ and $l>3$ be such that  $l\equiv 1\;(\Mod\;3)$. Then, $n(k,l)\geq \frac{2k(l-1)}{3}+2$.
	\end{proposition}
	
	\begin{proof}
		Let $n=\frac{2k(l-1)}{3}+2$ and $\sigma \in A_n$. Note that $n_{\sigma}\leq \frac{m_{\sigma}}{2}\leq \frac{k}{3}(l-1)+1.$ If $n_{\sigma}=\frac{k}{3}(l-1)+1$, then $m_{\sigma}=\frac{2k(l-1)}{3}+2$. Thus $m_{\sigma}+n_{\sigma}=k(l-1)+3$ is odd, since either $k$ is even or $l$ is odd. This contradicts that $\sigma \in A_n$. We conclude that $n_{\sigma}\leq \frac{k}{3}(l-1)$. Now it follows that $\sigma \in \p(k,l;n)$ by \Cref{bounded_part_partitions_l=1(mod 3)} and we are done.
	\end{proof}
	
	Now we construct counterexamples in $A_{\frac{2k(l-1)}{3}+3}$, thereby showing that $n(k,l)\leq \frac{2k(l-1)}{3}+2.$ 

	\begin{proposition}\label{UB_for_l=1(mod 3)}
		Let $k\geq 2$ and $l>3$ be such that $l\equiv 1\;(\Mod\; 3)$. Let $n=\frac{2k(l-1)}{3}+3$. 
		\begin{enumerate}
			\item Suppose $k$ is even. Let $\sigma \in A_n$ be such that $\sigma$ is a disjoint product of $\frac{n-3}{2}$ many $2$-cycles and one $3$-cycle. Then $\sigma \notin \p(k,l;n)$. 
			\item Suppose $k$ is even. Let $\sigma \in A_{n+1}$ be such that and $\sigma$ 
			is a disjoint product of $\frac{n+1}{2}$ many $2$-cycles. Then $\sigma \notin \p(k,l;n+1)$. 
			\item Suppose $k$ is odd. Let $\sigma \in S_n$ be such that $\sigma$ is a disjoint product of $\frac{n-3}{2}$ many $2$-cycles and one $3$-cycle. Then $\sigma \notin \p(k,l;n).$ 
			\item Suppose $k$ is odd. Let $\sigma \in S_{n+1}$ be such that $\sigma$ is a disjoint product of $\frac{n+1}{2}$ many $2$-cycles. Then $\sigma \notin \mathcal{P}(k,l;n+1).$ 
		\end{enumerate}
	In particular, $n(k,l)\leq \frac{2k(l-1)}{3}+2$.	
	\end{proposition}

	\begin{proof}
		We prove the result by induction on $k$. For $k=2$, we prove (1) and (2). We have $n=\frac{4(l-1)}{3}+3=\frac{4l+2}{3}+1$. Let $\sigma \in A_{n}$ be a disjoint product of $\frac{2(l-1)}{3}$ many $2$-cycles and one $3$-cycle. Clearly $m_{\sigma}=\frac{4l+2}{3}+1$ and $n_{\sigma}=\frac{2(l-1)}{3}+1$. We have $2l-m_{\sigma}-n_{\sigma}=2l-\frac{4l+2}{3}-\frac{2l-2}{3}-2=-2 <0$. 
		Thus $\sigma \notin \mathcal{P}(2,l;n)$ by \Cref{general_product_condition}.  
		Now suppose $\sigma \in A_{n+1}$ be a disjoint product of $\frac{2(l-1)}{3}+2$ many $2$-cycles. Then $m_{\sigma}=\frac{4l+2}{3}+2$. We have $2l-m_{\sigma}-n_{\sigma}=2l-\frac{4l+2}{3}-\frac{2l-2}{3}-4=-4 <0$. 
		Thus $\sigma \notin \mathcal{P}(2,l;n)$ by \Cref{general_product_condition}.

		We now prove statement (3) and (4) for $k=3$. Let $n=\frac{6(l-1)}{3}+3=2l+1$. 
		Let $\sigma \in S_{n}$ be a disjoint product of $l-1$ many $2$-cycles and one $3$-cycle. Then, $3l-m_{\sigma}-n_{\sigma}=3l-2l-1-l=-1 <0$. 
		Thus $\sigma \notin \mathcal{P}(k,l;n)$ by \Cref{general_product_condition}. 
		Now let $\sigma \in S_{n+1}$ be a disjoint product of $l+1$ many $2$-cycles. 
		Then, $3l-m_{\sigma}-n_{\sigma}=3l-2l-2-l-1=-3<0$. 
		Thus $\sigma \notin \mathcal{P}(k,l;n)$ by \Cref{general_product_condition}.
		
		Let us assume that the result holds for every $k'$ such that $2\leq k'\leq k-2$. We show that the result holds for $k$. For any $k\geq 4$, we first note the following:
		
		\begin{enumerate}[label=(\alph*)]
			\item If $\sigma$ is as in (1), then $kl-m_{\sigma}-n_{\sigma}= kl-\frac{2k(l-1)}{3}-3-\frac{k(l-1)}{3}-1=k-4< k-2$.
			\item If $\sigma$ is as in (2), then $kl-m_{\sigma}-n_{\sigma}= kl-\frac{2k(l-1)}{3}-4-\frac{k(l-1)}{3}-2=k-6<k-4$.
			\item If $\sigma$ is as in (3), then $kl-m_{\sigma}-n_{\sigma}= kl-\frac{2k(l-1)}{3}-3-\frac{k(l-1)}{3}-1=k-4< k-2$.
			\item If $\sigma$ is as in (4), then $kl-m_{\sigma}-n_{\sigma}= kl-\frac{2k(l-1)}{3}-4-\frac{k(l-1)}{3}-2=k-6<k-2$.
		\end{enumerate}
		
		\noindent By \Cref{indecomposable}, to show that $\sigma$ cannot be written as a product of $k$-many $l$-cycles, it is enough to prove that in the case (1), (3), and (4) $\sigma$ is $k$-indecomposable and in case (2), $\sigma$ is even-$k$-indecomposable. We now show this in each case.
	    \medskip
		
	    \textbf{(1)} Assume on the contrary that there exist two disjoint permutations $\rho, \tau \in S_n$ such that $\sigma=\rho \tau$,  $\rho \in \p(k',l;m_{\rho})$ and $\tau \in \p(k-k',l; m_{\tau})$ for some $k'$, $2\leq k' \leq k-2$. 
		
		\medskip
		
		\textbf{\underline{Case A}:} Assume that $k'$ and $k-k'$ are even. Then $\rho,\tau \in A_n$. WLOG we may assume that $\rho$ is a disjoint product of $2$-cycles only. If $m_{\rho}\geq \frac{2k'(l-1)}{3}+4$ then $\rho \notin \p(k,l;m_{\rho})$ by induction hypothesis $(2)$ and \Cref{adding_even_many_two_cycles} (since (b) holds), a contradiction. Now let $m_{\rho}<\frac{2k'(l-1)}{3}+4$. Then $m_{\tau}=m_{\sigma}-m_{\rho}>\frac{2k(l-1)}{3}+3-\frac{2k'(l-1)}{3}-4.$ 
		Thus $m_{\tau}>\frac{2(k-k')(l-1)}{3}-1$. Since $m_{\tau}\equiv 3\;(\Mod\;4)$ and  $\frac{2(k-k')(l-1)}{3} \equiv 0\;(\Mod\;4)$, 
		we have $m_{\tau} \geq \frac{2(k-k')(l-1)}{3}+3.$ Thus, by induction hypothesis $(1)$ and \Cref{adding_even_many_two_cycles} (since (a) holds), we conclude that $\tau \notin \p(k-k',l;m_{\tau})$, once again a contradiction.
		
		\medskip
		
		\textbf{\underline{Case B}:} Assume that $k'$ and $k-k'$ are odd. WLOG we may assume that $\rho $ is a disjoint product of $2$-cycles only. If $m_{\rho}\geq \frac{2k'(l-1)}{3}+4$, then $\rho \notin \p(k',l;m_{\rho})$ by induction hypothesis $(4)$ and \Cref{adding_even_many_two_cycles} (since (d) holds), a contradiction. 
		Now let $m_{\rho}<\frac{2k'(l-1)}{3}+4$. Then $m_{\tau}=m_{\sigma}-m_{\rho}>\frac{2k(l-1)}{3}+3-\frac{2k'(l-1)}{3}-4.$ 
		Thus $m_{\tau}>\frac{2(k-k')(l-1)}{3}-1$. Further suppose that $l$ is odd. Note that $\rho, \tau \in A_n$. We have $m_{\tau}\equiv 3\;(\Mod\;4)$ and $\frac{2(k-k')(l-1)}{3}-1 \equiv 3\;(\Mod\;4)$. We have $m_{\tau} \geq \frac{2(k-k')(l-1)}{3}+3$. Now suppose that $l$ is even. Note that $\rho, \tau \in S_n\setminus A_n$. We have $m_{\tau}\equiv 1\;(\Mod\;4)$. Since $\frac{2(k-k')(l-1)}{3}-1 \equiv 1\;(\Mod\;4)$, it follows that $m_{\tau} \geq \frac{2(k-k')(l-1)}{3}+3$. Thus, by induction hypothesis $(3)$ and \Cref{adding_even_many_two_cycles} (since (c) holds), we conclude that $\tau \notin \mathcal{P}(k-k',l;n)$, once again a contradiction. Thus we conclude that $\sigma$ is $k$-indecomposable.

		\medskip
		
		\textbf{(2)}  On the contrary, suppose that there exist two disjoint $\rho, \tau \in A_{n+1}$ such that  $\rho \in \p(k',l;m_{\rho})$ and $\tau \in \p(k-k',l;m_{\tau})$ 
		for some even $k'$, $2\leq k' \leq k-2$.  If $m_{\rho}\geq \frac{2k'(l-1)}{3}+4$ then $\rho \notin \p(k,l;m_{\rho})$ by induction hypothesis $(2)$ and \Cref{adding_even_many_two_cycles} (since (b) holds), a contradiction. Thus we assume $m_{\rho}<\frac{2k'(l-1)}{3}+4$. Then $m_{\tau}=m_{\sigma}-m_{\rho}>\frac{2k(l-1)}{3}+4-\frac{2k'(l-1)}{3}-4$. Thus $m_{\tau}>\frac{2(k-k')(l-1)}{3}$. Since $m_{\tau}\equiv 0\;(\Mod\;4)$ and $\frac{2(k-k')(l-1)}{3} \equiv 0\;(\Mod\;4)$, 
		we have $m_{\tau} \geq \frac{2(k-k')(l-1)}{3}+4.$ Thus, by induction hypothesis (2) and \Cref{adding_even_many_two_cycles} (since (b) holds), we conclude that $\tau \notin \p(k-k',l;m_{\tau})$. Once again we get a contradiction and conclude that $\sigma$ is even-$k$-indecomposable.
		
		\medskip
		
		\textbf{(3)} On the contrary, suppose that there exist two disjoint permutations $\rho, \tau \in S_n$ such that $\sigma=\rho \tau$, and 
		$\rho \in P(k',l;m_{\rho})$ and $\tau \in P(k-k',l;m_{\tau})$ 
		for some $k'$, $2\leq k' \leq k-2$. WLOG we may assume that $k'$ is odd and $k-k'$ is even.
		
		\medskip
		
		\textbf{\underline{Case A}:} Assume that $\rho$ is a disjoint product of $2$-cycles only. If $m_{\rho}\geq \frac{2k(l-1)}{3}+4$, then by induction hypothesis $(4)$ and 
		\Cref{adding_even_many_two_cycles} (since (d) holds), $\rho \notin \p(k',l;m_{\rho})$, a contradiction. Thus we assume $m_{\rho}< \frac{2k(l-1)}{3}+4$. Then $m_{\tau}=m_{\sigma}-m_{\rho}>\frac{2(k-k')(l-1)}{3}-1$. Note that $m_{\tau}\equiv 3\;(\Mod\;4)$ and $\frac{2(k-k')(l-1)}{3} \equiv 0\;(\Mod\;4)$. Thus $m_{\tau}\geq \frac{2(k-k')(l-1)}{3}+3$. 
		Thus we have $\tau \notin \p(k-k',l;m_{\tau})$ by induction hypothesis $(1)$ and \Cref{adding_even_many_two_cycles} (since (a) holds). This is once again a contradiction.
		
		\medskip
		
		\textbf{\underline{Case B}:} Assume that $\rho$ is a disjoint product of one $3$ cycle and $n_{\rho}-1$ many $2$-cycles. If $m_{\rho}\geq \frac{2k'(l-1)}{3}+3$ then $\rho \notin \p(k',l;m_{\rho})$ by induction hypothesis $(3)$ and \Cref{adding_even_many_two_cycles} (since (c) holds), a contradiction. Suppose that $m_{\rho}< \frac{2k'(l-1)}{3}+3$. Then $m_{\tau}=m_{\sigma}-m_{\rho}>\frac{2(k-k')(l-1)}{3}$. 
		Note that $m_{\tau} \equiv 0\;(\Mod\;4)$ and $ \frac{2(k-k')(l-1)}{3} \equiv 0\;(\Mod\;4)$. Thus $m_{\tau} \geq \frac{2(k-k')(l-1)}{3}+4$ and $\tau \notin \p(k-k',l;m_{\tau})$ by induction hypothesis $(2)$  and \Cref{adding_even_many_two_cycles} (since (b) holds). We get a contradiction in this case too. We conclude that $\sigma$ is $k$-indecomposable.
		
		\medskip
		
		\textbf{(4)} Assume on the contrary that there exist disjoint permutations $\rho, \tau \in S_n$ such that $\sigma=\rho \tau$, $\rho \in P(k',l;m_{\rho})$ and $\tau \in P(k-k',l;m_{\tau})$ for some $k'$, $2\leq k' \leq k-2$. WLOG we may assume that $k'$ is odd and $k-k'$ is even. If $m_{\rho}\geq \frac{2k'(l-1)}{3}+4$, then $\rho\notin \p(k',l;m_{\rho})$ by induction hypothesis $(4)$ and  \Cref{adding_even_many_two_cycles} (since (d) holds), a contradiction.
		Assume that $m_{\rho}<\frac{2k'(l-1)}{3}+4$. Then $m_{\tau}=m_{\sigma}-m_{\rho}>\frac{2(k-k')(l-1)}{3}$. Note that $m_{\tau}\equiv 0\;(\Mod\;4)$ and $\frac{2(k-k')(l-1)}{3}\equiv 0\;(\Mod\;4)$. Thus $m_{\tau} \geq \frac{2(k-k')(l-1)}{3}+4$ and $\tau \notin \p(k-k',l;m_{\tau})$ by induction hypothesis $(2)$ and \Cref{adding_even_many_two_cycles} (since (b) holds). Thus we get a contradiction once again and conclude that $\sigma$ is $k$-indecomposable.
		
		\medskip
		
		The fact that $n(k,l)\leq \frac{2k(l-1)}{3}+2$ follows from (1) and from (3), since if $\sigma$ is as in (3) then $\sigma \in A_n$ when $l$ is odd.
		
	\end{proof}
	
	\noindent From \Cref{newvalues_k=2_k=3_k=4}, Proposition~\ref{lowerbound_n(k,l)_l=1(mod 3)} and \Cref{UB_for_l=1(mod 3)}, we conclude the following:
	$$\textbf{\Cref{maintheorem} holds when $k\geq 2$ and  $l\equiv 1\;(\Mod\; 3)$.}$$
	
	
	\section{Exact value of $n(k,l)$ when $k$ is even and $l\equiv 2\;(\Mod\;3)$}\label{Proof_of_main_theorem_k_even_l=2(mod 3)}
	
	In this section we prove \Cref{maintheorem} when $k$ is even and $l\equiv 2\;(\Mod\; 3)$. 
	
	\begin{lemma}\label{general_decomposibility_l(2mod3)_k_even}
		Let $k\geq 6$ be even and $l>3$ be such that $l\equiv 2\;(\Mod\; 3)$. Let $\sigma \in A_{\frac{k(2l-1)}{3}+1}$ be such that $n_{\sigma}\geq \frac{k(l-2)}{3}+2$. Then, $\sigma$ can be written as a product of two disjoint permutations $\rho$ and $\tau$ such that one of the following holds:
		\begin{enumerate}
			\item  $\rho\in A_{\frac{4l+1}{3}-\epsilon}$ and $\tau \in A_{\frac{(k-2)(2l-1)}{3}+\epsilon},$
			
			\item $\rho \in A_{\frac{8l-1}{3}-\epsilon}$ and $\tau \in A_{\frac{(k-4)(2l-1)}{3}+\epsilon},$
		\end{enumerate}
		where $\epsilon \in \{0,1\}$.
	\end{lemma}
	
	\begin{proof}
		We write $l=3M+2$ where $M\geq 1$. Note that $m_{\sigma}\leq \frac{k(2l-1)}{3}+1=k(2M+1)+1$ and $n_{\sigma}\geq kM+2.$ We also have $\frac{4l+1}{3}=4M+3$ and $\frac{8l-1}{3}=8M+5$. Note that we will write $n_i$ to mean $n_i(\sigma)$ throughout the proof. Now we construct $\rho$ by defining it to be a certain product of cycles occurring in $dcd*(\sigma)$. Note that it is enough to construct $\rho$ in the desired $A_n$ as in (1) or (2) as $\tau$ can then be chosen such that $dcd*(\tau)=dcd*(\sigma)\setminus dcd*(\rho).$
		
		Observe that one of $4M+2, 4M+3,8M+5$ is divisible by 3. Let $\alpha$ be the number among them  which is divisible by 3. Thus if $n_3\geq \frac{8M+5}{3}$, we can choose $\rho$ to be the product of $\alpha$ many 3-cycles from $dcd*(\sigma)$ when $\alpha$ is either $4M+2$ or $4M+3$, whence it follows that $\rho\in A_{\frac{4l+1}{3}-\epsilon}$ where $\epsilon=1$ and $\epsilon=0$ respectively. If $\alpha=8M+5$, we can choose $\rho$ to be the product of $\alpha$ many $3$-cycles from $dcd*(\sigma)$, whence $\rho \in A_{\frac{8l-1}{3}}$. So assume that $n_3\leq \frac{8M+5}{3}$. Now we get a least estimate for $n_2$ using \Cref{cycle_bound_inequality}. Putting $j=2$ and $i=3$ in \Cref{cycle_bound_inequality}, we get
		\begin{equation}\label{eqn_prop3.3_1}
			2n_2 \geq 4n_{\sigma}-m_{\sigma}-n_3\geq 4(kM+2)-(2kM+k+1)-\frac{8M+5}{3}
		\end{equation}
		$\displaystyle \hspace{2.7 cm}= \frac{k(6M-3)-8M+16}{3}\geq \frac{28M-2}{3}$.
		
		\vspace{1 mm}
		\noindent The last inequality follows since $M\geq 1$ and $k\geq 6$. This implies that $n_2\geq \frac{14M-1}{3}$. Note that $\frac{14M-1}{3}\geq 2M$. If we assume that $n_3\geq 1$, then we can choose $\rho$ to be the product of $2M$ many 2-cycles and a 3-cycle from $dcd*(\sigma)$, whence $\rho \in A_{\frac{4l+1}{3}}$. Thus we now assume that $n_3=0$. Doing the same computation as in \Cref{eqn_prop3.3_1}, with $n_3=0$, we get an improved least estimate of $n_2$ as follows:
		$$2n_2\geq 4n_{\sigma}-m_{\sigma}-n_3\geq 4(kM+2)-2kM-k-1= k(2M-1)+7\geq 12M+1.$$
		The last inequality follows since $M\geq 1$ and $k\geq 6$.
		This implies $n_2\geq 6M+1$. Thus, we can choose $\rho$ to be the product of $4M+2$ many 2-cycles, whence $\rho \in A_{\frac{8l-1}{3}-1}.$ The proof is now complete.
	\end{proof}
	
	\begin{proposition}\label{n(k,l_lowerbound_k_even_l(2 mod 3)}
		Let $l>3$ and $l\equiv 2\;(\Mod\;3)$. Let $k\geq 2$ be even. Then $n(k,l)\geq \frac{k(2l-1)}{3}+1$.
	\end{proposition}
	
	\begin{proof}
		We use induction on $k$. By \Cref{newvalues_k=2_k=3_k=4}, the result holds for $k=2$ and $k=4$. Suppose that the result holds for all even $k'$ such that $k'\leq k-2$. We show that the result holds for $k$. Let $n=\frac{k(2l-1)}{3}+1$. Suppose $\sigma \in A_n$. If $n_{\sigma}\leq \frac{k(l-2)}{3}+1$, by \Cref{bounded_part_partitions_l=2(mod 3)}, we conclude that $\sigma \in \p(k,l;n)$. Thus we assume $n_{\sigma}\geq \frac{k(l-2)}{3}+2$. By \Cref{general_decomposibility_l(2mod3)_k_even}, we conclude that $\sigma$ can be written as a disjoint product of two permutations $\rho$ and $\tau$, where either $\rho\in A_{\frac{4l+1}{3}-\epsilon}$ and $\tau \in A_{\frac{(k-2)(2l-1)}{3}+\epsilon}$ , or $\rho\in A_{\frac{8l-1}{3}-\epsilon}$ and $\tau \in A_{\frac{(k-4)(2l-1)}{3}+\epsilon}$; $\epsilon \in \{0,1\}$. Note that in the first case, by \Cref{newvalues_k=2_k=3_k=4}, $\rho \in \p(2,l;\frac{4l+1}{3}-\epsilon)$ and by induction hypothesis $\tau \in \p(k-2,l; \frac{(k-2)(2l-1)}{3}+\epsilon)$. It follows that $\sigma \in \p(k,l;n)$. In the later case, by \Cref{newvalues_k=2_k=3_k=4}, $\rho \in \p(4,l; \frac{8l-1}{3}-\epsilon)$ and by induction hypothesis $\tau \in \p(k-4,l; \frac{(k-4)(2l-1)}{3}+\epsilon)$, from which it follows that $\sigma \in \p(k,l;n)$.
	\end{proof}
	
	Now we proceed to show that $n(k,l)\leq \frac{k(2l-1)}{3}+1$. To do that we find a permutation $\sigma$ in $A_{\frac{k(2l-1)}{3}+2}$ for each $k$ even such that $\sigma \notin \p(k,l;\frac{k(2l-1)}{3}+2)$.

	\begin{proposition}\label{UB_for_odd_l}
		Let $k\geq 2$ be even and $l\equiv 2 \;(\Mod\;3)$. Let $n= \frac{k(2l-1)}{3}+2$. The following statements hold:
		\begin{enumerate}
			\item If $k\equiv 2 \;(\Mod\;4)$ and $\sigma \in A_n$ is a disjoint product of $\frac{n}{2}$ many $2$-cycles then $\sigma \notin \p(k,l;n)$.
			
			\item If $k\equiv 2\;(\Mod\;4)$ and $\sigma \in A_{n+2}$ is a  disjoint product of $\frac{n}{2}-1$ many $2$-cycles and one $4$-cycle then $\sigma \notin \p(k,l;n+2)$.
			
			\item If $k\equiv 0\;(\Mod\;4)$ and $\sigma \in A_n$ is a disjoint product of $\frac{n}{2}-2$ many $2$-cycles and one $4$-cycle then $\sigma \notin \p(k,l;n)$.
			
			\item If $k\equiv 0 \;(\Mod\;4)$ and $\sigma \in A_{n+2}$ is a disjoint product of $\frac{n}{2}+1$ many $2$-cycles then $\sigma \notin \p(k,l;n+2)$.
		\end{enumerate}
		In particular, $n(k,l)\leq \frac{k(2l-1)}{3}+1$.
	\end{proposition}
	
	\begin{proof}
		 It is easy to see that the permutation considered in each is an even permutation. We prove the result by induction on $k$. Let $k=2$. Then $n=\frac{4l+4}{3}$. 
		Let $\sigma \in A_n$ be a disjoint product of $\frac{n}{2}$ many $2$-cycles. 
		Then $2l-m_{\sigma}-n_{\sigma}= 2l-\frac{4l+4}{3}-\frac{2l+2}{3}=-2<0$. Hence $\sigma \notin \p(2,l;n)$ by \Cref{general_product_condition}. Now let $\sigma \in A_{n+2}$ be a disjoint product of $\frac{n}{2}-1$ many $2$-cycles and one $4$-cycle. Then $2l-m_{\sigma}-n_{\sigma}=2l-\frac{4l+4}{3}-2-\frac{2l+2}{3}=-4$. Hence $\sigma \notin \p(2,l;n+2)$ by \Cref{general_product_condition}.
		
		\medskip
		
		Let $k=4$. Then $n=\frac{4(2l-1)}{3}+2=\frac{8l+2}{3}$. Let $\sigma \in A_n$ be disjoint product of $\frac{n}{2}-2$ many $2$-cycles and one $4$-cycle. Suppose that $\sigma \in \p(4,l;n)$.  We first show that $\sigma$ is $(2,2)$ indecomposable. Assume, on the contrary, that there exist two disjoint permutations $\rho ,\tau \in A_n$ such that $\sigma=\rho \tau$, $\rho \in \p(2,l;m_{\rho})$ and $\tau \in \p(2,l;m_{\tau})$. WLOG we may assume that $\rho$ is a disjoint product of $\frac{n_{\rho}}{2}$ many $2$ cycles and $\tau $ is a product of $\frac{m_{\tau}-4}{2}$ many $2$-cycles and one $4$-cycle. Note that $m_{\rho}=2n_{\rho}$ and $m_{\tau}=2n_{\tau}+2$. By \Cref{general_product_condition},  $2l-m_{\rho}-n_{\rho} \geq 0$ which implies $m_{\rho}\leq \frac{4l-2}{3}$. Similarly, $2l-m_{\tau}-n_{\tau} \geq0$ which implies $m_{\tau}\leq \frac{4l-2}{3}$. Thus we get that $\frac{8l+2}{3}=m_{\sigma}=m_{\rho}+m_{\tau} \leq \frac{8l-4}{3}$, which is a contradiction. Hence $\sigma$ is $(2,2)$ indecomposable. We note that $4l-m_{\sigma}-n_{\sigma}=4l-\frac{8l+2}{3}-\frac{4l+1}{3}+1=0$, which is a contradiction by \Cref{indecomposable}. Thus $\sigma \notin \p(4,l;n)$.
		
		Suppose now that $\sigma \in A_{n+2}$ be disjoint product of $\frac{n}{2}+1$ many $2$-cylces. We note that $4l-m_{\sigma}-n_{\sigma }=4l-\frac{8l+2}{3}-2-\frac{4l+1}{3}-1=-4<0$. Thus by \Cref{general_product_condition}, $\sigma \notin \p(4,l;n+2)$. Hence the lemma holds for $k=2,4$. Let $k \geq 6$ and suppose that the result holds for all even $k'\leq k-2$. For each $k\geq 6$, $k$ even and $n=\frac{k(2l-1)}{3}+2$, we first note the following:
		
		\begin{enumerate}[label=(\alph*)]
			\item If $\sigma$ is as in (1), then $kl-m_{\sigma}-n_{\sigma}= kl-\frac{k(2l-1)}{3}-2-\frac{k(2l-1)}{6}-1< k-4$.
			\item If $\sigma$ is as in (2), then $kl-m_{\sigma}-n_{\sigma}= kl-\frac{k(2l-1)}{3}-4-\frac{k(2l-1)}{6}-1< k-4$.
			\item If $\sigma$ is as in (3), then $kl-m_{\sigma}-n_{\sigma}= kl-\frac{k(2l-1)}{3}-2-\frac{k(2l-1)}{6}< k-4$.
			\item If $\sigma$ is as in (4), then $kl-m_{\sigma}-n_{\sigma}= kl-\frac{k(2l-1)}{3}-4-\frac{k(2l-1)}{6}-2< k-4$.
		\end{enumerate}
		From the above and using \Cref{indecomposable}, it is enough to prove that $\sigma$ is even-$k$-indecomposable in each case, which we establish now.
		
		\medskip

		\textbf{(1)} Assume that $k\equiv 2 \;(\Mod\;4)$. Suppose, on the contrary, that there exists even $k'$, $2 \leq k'\leq k-2$  and there exist two disjoint permutations $\rho ,\tau \in A_n$ such that $\sigma=\rho \tau$ and $\rho \in \p(k',l;m_{\rho})$ and $\tau \in \p(k-k',l;m_{\tau})$. Then $\rho $ (resp. $\tau$) is a product of 
		$\frac{m_{\rho }}{2}$ (resp. $\frac{m_{\tau }}{2}$) many $2$-cycles. Note that $n_{\rho}=\frac{m_{\rho}}{2}$ and $n_{\tau}=\frac{m_{\tau}}{2}$. WLOG we may assume that $k'\equiv 2\;(\Mod\;4)$ and $(k-k')\equiv 0\;(\Mod\;4)$. Assume that $m_{\rho}= \frac{k'(2l-1)}{3}+2$. By induction hypothesis (1), we conclude that $\rho \notin \p(k',l;m_{\rho})$, which is a contradiction. Further if  $m_{\rho}>\frac{k'(2l-1)}{3}+2$, since (a) holds, applying \Cref{adding_even_many_two_cycles} we conclude that $\rho \notin \p(k',l;m_{\rho})$, once again a contradiction.
		
		Now assume that $m_{\rho} < \frac{k'(2l-1)}{3}+2$. Since $m_{\rho}$ and $\frac{k'(2l-1)}{3}+2$ are both divisible by $4$, we have $m_{\rho}\leq \frac{k'(2l-1)}{3}-2$. 
		We have $m_{\tau}=m_{\sigma}-m_{\rho}\geq  \frac{(k-k')(2l-1)}{3}+4 $. 
		By induction hypothesis (4) and \Cref{adding_even_many_two_cycles} (since (d) holds), we conclude that $\tau \notin \p(k-k',l;m_{\tau})$, which is a contradiction. Hence $\sigma$ is even-$k$-indecomposable as required.
		\medskip
		
		\textbf{(2)} Assume that $k\equiv 2 \;(\Mod\;4)$. Suppose, on the contrary, that there exists even $k'$, $2 \leq k'\leq k-2$ and there exist two disjoint permutations $\rho ,\tau \in A_{n+2}$ such that $\sigma=\rho \tau$ and $\rho \in P(k',l;m_{\rho})$ and $\tau \in P(k-k',l;m_{\tau})$. WLOG we may assume that $\rho$ is a product of $\frac{m_{\rho}}{2}$ many $2$-cycles and $\tau$ is a product of $\frac{m_{\tau}-4}{2}$ many $2$-cycles and one $4$-cycle. Either $(i)$ $k'\equiv 0\;(\Mod\;4)$ and $k-k'\equiv 2\;(\Mod\;4)$ or 
		$(ii)$ $k-k' \equiv 0\;(\Mod\;4)$ and $k'\equiv 2\;(\Mod\;4)$.
		
		Assume that $(i)$ holds. If $m_{\rho}\geq \frac{k'(2l-1)}{3}+4$ 
		then $\rho \notin \p(k',l;m_{\rho})$ by induction hypothesis (4) and by \Cref{adding_even_many_two_cycles} (since (d) holds), which is a contradiction. Now assume that $m_{\rho}<\frac{k'(2l-1)}{3}+4$, i.e., $m_{\rho}\leq \frac{k'(2l-1)}{3}$. Consequently, $m_{\tau}=m_{\sigma}-m_{\rho} \geq \frac{(k-k')(2l-1)}{3}+4$. Notice that  $\tau \notin \p(k-k',l,m_{\tau})$ by induction hypothesis (2) and  \Cref{adding_even_many_two_cycles} (since (b) holds), which is a contradiction. 
		
		Now assume that $(ii)$ holds. 
		If $m_{\rho}\geq \frac{k'(2l-1)}{3}+2$ then  $\rho \notin \p(k',l;m_{\rho})$ by induction hypothesis (1) and \Cref{adding_even_many_two_cycles} (since (a) holds), which is a contradiction. Assume that $m_{\rho} < \frac{k'(2l-1)}{3}+2$. Then $m_{\rho}\leq \frac{k'(2l-1)}{3}-2$. We get $m_{\tau}=m_{\sigma}-m_{\rho}\geq \frac{(k-k')(2l-1)}{3}+6$. Thus by induction hypothesis (3) and \Cref{adding_even_many_two_cycles} (since (c) holds), we conclude that $\tau\notin \p(k-k',l;m_{\tau})$, which is a contradiction. We conclude that $\sigma$ is even-$k$-indecomposable in this case too.
		
		\medskip
		
		\textbf{(3)} Assume that $k\equiv 0\;(\Mod\;4)$. Suppose, on the contrary, that there exists even $k'$, $2 \leq k'\leq k-2$ and there exist two disjoint permutations $\rho ,\tau \in A_n$ such that $\sigma=\rho \tau$ and $\rho \in \p(k',l;m_{\rho})$ and $\tau \in \p(k-k',l;m_{\tau})$. WLOG we may assume that $\rho$ is a product of $\frac{m_{\rho}}{2}$ many $2$-cycles and $\tau$ is a product of $\frac{m_{\tau}-4}{2}$ many $2$-cycles and one $4$-cycle. Either $(i)$ $k',k-k'\equiv 0\;(\Mod\;4)$, or  $(ii)$ $k', k-k' \equiv 2\;(\Mod\;4)$.
		
		Assume that $(i)$ holds. If $m_{\rho}\geq \frac{k'(2l-1)}{3}+4$ then $\rho \notin \p(k',l;m_{\rho})$ by induction hypothesis (4), and by \Cref{adding_even_many_two_cycles} (since (d) holds), which is a contradiction. 
		Now assume that $m_{\rho}<\frac{k'(2l-1)}{3}+4$. Then $m_{\rho}\leq \frac{k(2l-1)}{3}$. We have $m_{\tau}=m_{\sigma}-m_{\rho}\geq \frac{(k-k')(2l-1)}{3}+2$. We conclude that $\tau \notin \p(k-k',l;m_{\tau})$ by induction hypothesis (3), and by \Cref{adding_even_many_two_cycles} (since (c) holds), which is a contradiction. 
		
		Assume that $(ii)$ holds. If $m_{\rho}\geq \frac{k'(2l-1)}{3}+2$ then $\rho \notin P(k',l;m_{\rho})$ by induction hypothesis (1) and by \Cref{adding_even_many_two_cycles} (since (a) holds), which is a contradiction. Now assume that $m_{\rho}< \frac{k'(2l-1)}{3} +2$. Then $m_{\rho}\leq \frac{k'(2l-1)}{3}-2 $. We have $m_{\tau}=m_{\sigma}-m_{\rho}\geq \frac{(k-k')(2l-1)}{3}+4$. Now by induction hypothesis (2) and  \Cref{adding_even_many_two_cycles} (since (b) holds), we conclude that $\tau \notin \p(k-k',l,m_{\rho})$, which is a contradiction. We once again conclude that $\sigma$ is even-$k$-indecomposable in this case.
		\medskip
		
		\textbf{(4)} Assume that $k\equiv 0\;(\Mod\;4)$ and $\sigma \in P(k,l;n+2)$. Assume, on the contrary, that there exists even $k'$, $2 \leq k'\leq k-2$ and there exist two disjoint permutations $\rho ,\tau \in A_{n+2}$ such that $\sigma=\rho \tau$ and $\rho \in \p(k',l;m_{\rho})$ and $\tau \in \p(k-k',l;m_{\tau})$. Then $\rho$ is a product of $\frac{m_{\rho}}{2}$ many $2$-cycles and $\tau$ is a product of $\frac{m_{\tau}}{2}$ many $2$-cycles. Either $(i)$ $k',k-k'\equiv 0\;(\Mod\;4)$ or $(ii)$ $k', k-k' \equiv 2\;(\Mod\;4)$.
		
		Assume that $(i)$ holds. If $m_{\rho}\geq \frac{k'(2l-1)}{3}+4$, then by induction hypothesis (4) and  \Cref{adding_even_many_two_cycles} (since (d) holds), we conclude that $\rho \notin \p(k',l,m_{\rho})$, which is a contradiction. 
		Now assume that $m_{\rho}< \frac{k'(2l-1)}{3}+4$. Then  $m_{\rho}\leq \frac{k'(2l-1)}{3}$. We have $m_{\tau}=m_{\sigma}-m_{\rho} \geq \frac{(k-k')(2l-1)}{3}+4$. By induction hypothesis (4), and by \Cref{adding_even_many_two_cycles} (since (d) holds), we conclude $\tau \notin \p(k-k',l;m_{\tau})$, which is a contradiction.
		
		Assume that $(ii)$ holds. If $m_{\rho}\geq \frac{k'(2l-1)}{3}+2$, then by induction hypothesis (1) and \Cref{adding_even_many_two_cycles} (since (a) holds), we conclude that $\rho \notin \p(k',l,m_{\rho})$, which is a contradiction. 
		Now assume that $m_{\rho}< \frac{k'(2l-1)}{3}+2$. Then $m_{\rho}\leq \frac{k'(2l-1)}{3}-2$. We have $m_{\tau}=m_{\sigma}-m_{\rho} \geq \frac{(k-k')(2l-1)}{3}+6$. Once again by induction hypothesis (1) and \Cref{adding_even_many_two_cycles} (since (a) holds), we conclude $\tau \notin \p(k-k',l;m_{\tau})$, which is a contradiction. Thus $\sigma$ is even-$k$-indecomposable in this case as well and the proof is complete.
	\end{proof}
	
	\noindent From \Cref{newvalues_k=2_k=3_k=4}, Proposition~\ref{n(k,l_lowerbound_k_even_l(2 mod 3)} and \Cref{UB_for_odd_l}, we conclude the following:
	
	$$\textbf{\Cref{maintheorem} holds when $k\geq 2$ is even and $l\equiv 2\;(\Mod\; 3)$.}$$
	
	\section{Complete Proof of \Cref{maintheorem}}\label{complete_proof}
	
	In this section, we complete the proof of \Cref{maintheorem} by finding the exact value of $n(k,l)$ when $k$ is odd and $l\equiv 2\;(\Mod \; 3)$. We start with $k=5$.
	
	\begin{lemma}\label{lowerbound_k=5}
		Let $l>3$ be such that $l\equiv 2\;(\Mod\; 3)$ and $l$ is odd. Then $n(5,l)\geq \frac{10l-5}{3}.$
	\end{lemma}
	
	\begin{proof}
		Let $\sigma \in A_{\frac{10l-5}{3}}.$ When $n_{\sigma}\leq \frac{k(l-2)}{3}+1=\frac{5(l-2)}{3}+1$, then by \Cref{bounded_part_partitions_l=2(mod 3)}, $\sigma \in \p(5,l;n)$. Thus we now assume that $n_{\sigma}\geq \frac{5(l-2)}{3}+2$. Write $l=3M+2$. Note that $M\geq 1$ is odd as $l$ is odd. Thus $\sigma\in A_{10M+5}$ with $n_{\sigma}\geq 5M+2$. Note that in this case it is clear that $n_{\sigma}=5M+2$, for if $n_{\sigma}\geq 5M+3$, then $m_{\sigma}\geq 10M+6$, a contradiction. Further since $n_{\sigma}=5M+2$, $m_{\sigma}\geq 10M+4$. If $m_{\sigma}=10M+4$, then $m_{\sigma}+n_{\sigma}=15M+6$ which is odd, since $M$ is odd. This is a contradiction as $\sigma \in A_n$. Thus we get $m_{\sigma}=10M+5$. Note that this forces $\sigma$ to be a disjoint product of $5M+1$ many 2-cycles and a 3-cycle. If we choose $\rho$ to be a product of $3M+1$ many 2-cycles from $dcd*(\sigma)$ and $\tau$ to be the product of cycles in $dcd*(\sigma)\setminus dcd*(\rho)$, then it is easily seen that $\rho \in A_{6M+2}$ and $\tau\in A_{4M+3}$ and $\sigma=\rho\tau$. Since $2l=6M+4$ and $\frac{4l+1}{3}=4M+3$, by \Cref{newvalues_k=2_k=3_k=4}, we conclude that $\rho \in \p(3,l; 2l-2)$ and $\tau \in \p(2,l; \frac{4l+1}{3})$, whence $\sigma \in \p(5,l;\frac{10l-5}{3})$. This completes the proof.
	\end{proof}

	\begin{lemma}\label{decomposability_lemma_k=3(mod 4)_l_odd_2(mod3)}
		Let $k\geq 7$ be odd, $k\equiv 3\;(\Mod\; 4)$ and $l>3$ be such that $l\equiv 2\;(\Mod\;3)$. Let $n=\frac{k(2l-1)}{3}+1$. Suppose $\sigma \in A_n$ be such that $n_{\sigma}\geq \frac{k(l-2)}{3}+2$. Then, $\sigma$ can be written as a product of two disjoint permutations $\rho$ and $\tau$ such that one of the following holds:
		\begin{enumerate}
			\item  $\rho\in A_{\frac{4l+1}{3}}$ and $\tau \in A_{\frac{(k-2)(2l-1)}{3}},$
			
			\item $\rho\in A_{\frac{8l-1}{3}-\epsilon}$ and $\tau \in A_{\frac{(k-4)(2l-1)}{3}+\epsilon},$
		\end{enumerate}
		where $\epsilon \in \{0,1\}$.
	\end{lemma}
	
	\begin{proof}
		We write $l=3M+2$ where $M\geq 1$. Note that $m_{\sigma}\leq \frac{k(2l-1)}{3}+1=k(2M+1)+1$ and $n_{\sigma}\geq kM+2.$ We also have $\frac{4l+1}{3}=4M+3$ and $\frac{8l-1}{3}=8M+5$. Note that we will write $n_i$ to mean $n_i(\sigma)$ throughout the proof. Now we construct $\rho$ by defining it to be a certain product of cycles occurring in $dcd*(\sigma)$. Note that it is enough to construct $\rho$ in the desired $A_n$ as in (1) or (2) as $\tau$ can then be chosen to be the product of all the cycles occuring in $dcd*(\sigma)\setminus dcd*(\rho)$. 
		
		Observe that one of $4M+3, 8M+4,8M+5$ is divisible by 3. Let $\alpha$ be the number among them  which is divisible by 3. Thus if $n_3\geq \frac{8M+5}{3}$, we can choose $\rho$ to be the product of $\alpha$ many 3-cycles from $dcd*(\sigma)$ when $\alpha$ is either $8M+4$ or $8M+5$, whence it follows that $\rho\in A_{\frac{8l-1}{3}-\epsilon}$ where $\epsilon=1$ and $\epsilon=0$ respectively. If $\alpha=4M+3$, we can choose $\rho$ to be the product of $\alpha$ many $3$-cycles from $dcd*(\sigma)$, whence $\rho \in A_{\frac{4l+1}{3}}$. So assume that $n_3\leq \frac{8M+5}{3}$. Now we get a least estimate for $n_2$ using Equation~\ref{cycle_bound_inequality}. Putting $j=2$ and $i=3$ in Equation~\ref{cycle_bound_inequality}, we get
		\begin{equation}\label{eqn_prop5.2_1}
		2n_2 \geq 4n_{\sigma}-m_{\sigma}-n_3\geq 4(kM+2)-(2kM+k+1)-\frac{8M+5}{3}
		\end{equation}
		$\displaystyle \hspace{2.7 cm} = \frac{k(6M-3)-8M+16}{3}\geq \frac{34M-5}{3}$.
		
		\vspace{1 mm}
		\noindent The last inequality follows since $M\geq 1$ and $k\geq 7$. Note that $\frac{34M-5}{3}\geq 4M$. If we assume that $n_3\geq 1$, then we can choose $\rho$ to be the product of $2M$-many 2-cycles and a 3-cycle from $dcd*(\sigma)$, whence $\rho \in A_{\frac{4l+1}{3}}$. Thus we now assume that $n_3=0$. Doing the same computation as in Equation~\ref{eqn_prop5.2_1} with $n_3=0$, we get an improved least estimate of $n_2$ as follows:
		$$2n_2\geq 4n_{\sigma}-m_{\sigma}-n_3\geq 4(kM+2)-2kM-k-1= k(2M-1)+7\geq 14M.$$
		The last inequality follows since $M\geq 1$ and $k\geq 7$.
		This implies $n_2\geq 7M$. Thus, we can choose $\rho$ to be the product of $4M+2$ many 2-cycles, whence $\rho \in A_{\frac{8l-1}{3}-1}.$ The proof is now complete.
	\end{proof}
	
	\begin{lemma}\label{decomposability_lemma_k=1(mod 4)_l_odd_2(mod3)}
		Let $k\geq 9$ be odd, $k\equiv 1\;(\Mod\; 4)$ and $l>3$ be such that $l\equiv 2\;(\Mod\;3)$. Let $n=\frac{k(2l-1)}{3}$. Suppose $\sigma \in A_n$ be such that $n_{\sigma}\geq \frac{k(l-2)}{3}+2$. Then, $\sigma$ can be written as a product of two disjoint permutations $\rho$ and $\tau$ such that one of the following holds:
		\begin{enumerate}
			\item  $\rho\in A_{\frac{4l+1}{3}-\epsilon}$ and $\tau \in A_{\frac{(k-2)(2l-1)}{3}-1+\epsilon},$
			
			\item $\rho\in A_{\frac{8l-1}{3}-1}$ and $\tau \in A_{\frac{(k-4)(2l-1)}{3}},$
		\end{enumerate}
		where $\epsilon \in \{0,1,2\}$.
	\end{lemma}

	\begin{proof}
		We write $l=3M+2$ where $M\geq 1$. We get $M$ is odd since $l$ is odd. Let $\sigma \in A_{\frac{k(2l-1)}{3}}$ such that $n_{\sigma}\geq \frac{k(l-2)}{3}+2=kM+2$. Note that $m_{\sigma}\leq k(2M+1)$. Further $\frac{4l+1}{3}=4M+3$ and $\frac{8l-1}{3}=8M+5$. Note that we will write $n_i$ to mean $n_i(\sigma)$ throughout the proof. Now we construct $\rho$ by defining it to be a certain product of cycles occurring in $dcd*(\sigma)$. Note that it is enough to construct $\rho$ in the desired $A_n$ as in (1) or (2) as $\tau$ can then be chosen to be the product of all the cycles occuring in $dcd*(\sigma)\setminus dcd*(\rho).$
		
		Observe that one of $4M+1, 4M+2, 4M+3$ is divisible by 3. Let $\alpha$ be the number among them  which is divisible by 3. Thus if $n_3\geq \frac{4M+3}{3}$, we can choose $\rho$ to be the product of $\alpha$ many 3-cycles from $dcd*(\sigma)$ where $\alpha$ is either $4M+1$ or, $4M+2$ or $4M+3$, whence it follows that $\rho\in A_{\frac{4l+1}{3}-\epsilon}$ where $\epsilon=2$, $\epsilon=1$ and $\epsilon=0$ respectively. So assume that $n_3\leq \frac{4M+3}{3}$. Now we get a least estimate for $n_2$ using Equation~\ref{cycle_bound_inequality}. Putting $j=2$ and $i=3$ in Equation~\ref{cycle_bound_inequality}, we get
		\begin{equation}\label{eqn_prop5.3_1}
		2n_2 \geq 4n_{\sigma}-m_{\sigma}-n_3\geq 4(kM+2)-(2kM+k)-\frac{4M+3}{3}
		\end{equation}
		$\displaystyle \hspace{2.7 cm}= \frac{k(6M-3)-4M+21}{3}\geq \frac{50M-6}{3}$.
		
		\noindent The last inequality follows since $k\geq 9$. Thus we get $n_2\geq \frac{25M-3}{3}$ which implies that $n_2\geq 4M+2$. Thus we can choose $\rho$ to be the product of $4M+2$ many 2-cycle, whence $\rho \in A_{\frac{8l-1}{3}-1}$. This completes the proof.
	\end{proof}
	
	\begin{proposition}\label{lower_bound_for_k_odd_l_odd_l=2(mod3)}
		Let $k\geq 3$ be odd and $l>3$ be such that $l$ is odd and $l\equiv 2\;(\Mod\; 3)$. Then
		\[n(k,l)\geq 
		\begin{cases}
			\frac{k(2l-1)}{3} +1& \text{when $k\equiv 3\;(\Mod\; 4)$}\\
			\frac{k(2l-1)}{3} & \text{when $k\equiv 1\;(\Mod\; 4)$}.
		\end{cases}
		\]
	\end{proposition}

	\begin{proof}
		By \Cref{newvalues_k=2_k=3_k=4}, the result holds for $k=3$. By \Cref{lowerbound_k=5}, the result holds for $k=5$. We now use induction on $k$. Let the result hold for all odd $k'$ such that $k'\leq k-2$.
		
		\medskip
		
		\textbf{\underline{Case 1}:} Let $k\equiv 3\;(\Mod\; 4)$. Let $n=\frac{k(2l-1)}{3}+1$ and suppose that $\sigma \in A_{n}$. If $n_{\sigma}\leq \frac{k(l-2)}{3}$, then it follows from \Cref{bounded_part_partitions_l=2(mod 3)}, that $\sigma \in \p(k,l;n).$ Thus we assume $n_{\sigma}\geq \frac{k(l-2)}{3}+2.$ By \Cref{decomposability_lemma_k=3(mod 4)_l_odd_2(mod3)}, $\sigma$ can be written as a disjoint product of two permutation $\rho$ and $\tau$ such that either $\rho\in A_{\frac{4l+1}{3}}$ and $\tau \in A_{\frac{(k-2)(2l-1)}{3}}$ or, $\rho\in A_{\frac{8l-1}{3}-\epsilon}$ and $\tau \in A_{\frac{(k-4)(2l-1)}{3}+\epsilon}$, where $\epsilon \in \{0,1\}$. In the first case, by \Cref{newvalues_k=2_k=3_k=4}, $\rho \in \p(2,l;\frac{4l+1}{3})$. Further, by induction hypothesis, $\tau\in \p(k-2,l;\frac{(k-2)(2l-1)}{3})$ since $k-2\equiv 1\;(\Mod\; 4)$. We conclude that $\sigma \in \p(k,l;n)$. In the later case, once again by \Cref{newvalues_k=2_k=3_k=4}, $\rho \in \p(4,l; \frac{8l-1}{3}-\epsilon)$, and by induction hypothesis $\tau \in \p(k-4,l; \frac{(k-4)(2l-1)}{3}+\epsilon)$ since $k-4\equiv 3\;(\Mod\; 4)$. Once again we conclude that $\sigma \in \p(k,l;n)$.
		
		\medskip
		
		\textbf{\underline{Case 2}:} Let $k\equiv 1\;(\Mod\; 4)$. Let $n=\frac{k(2l-1)}{3}$ and suppose that $\sigma \in A_{n}$. If $n_{\sigma}\leq \frac{k(l-2)}{3}$, then it follows from \Cref{bounded_part_partitions_l=2(mod 3)}, that $\sigma \in \p(k,l;n).$ Thus we assume $n_{\sigma}\geq \frac{k(l-2)}{3}+2.$ By \Cref{decomposability_lemma_k=1(mod 4)_l_odd_2(mod3)}, $\sigma$ can be written as a disjoint product of two permutation $\rho$ and $\tau$ such that either $\rho\in A_{\frac{4l+1}{3}-\epsilon}$ and $\tau \in A_{\frac{(k-2)(2l-1)}{3}-1+\epsilon}$ ($\epsilon \in \{0,1,2\}$) or, $\rho\in A_{\frac{8l-1}{3}-1}$ and $\tau \in A_{\frac{(k-4)(2l-1)}{3}}$. In the first case, by \Cref{newvalues_k=2_k=3_k=4}, $\rho \in \p(2,l;\frac{4l+1}{3}-\epsilon)$. Further by induction hypothesis, $\tau\in \p(k-2,l;\frac{(k-2)(2l-1)}{3}-1+\epsilon)$ since $k-2\equiv 3\;(\Mod\; 4)$. We conclude that $\sigma \in \p(k,l;n)$. In the later case, once again by \Cref{newvalues_k=2_k=3_k=4}, $\rho \in \p(4,l; \frac{8l-1}{3}-1)$ and by induction hypothesis $\tau \in \p(k-4,l; \frac{(k-4)(2l-1)}{3})$. Once again we conclude that $\sigma \in \p(k,l;n)$.
		
		This completes the induction and proof is complete.
	\end{proof}

	We now proceed to determine the exact value of $n(k,l)$ when $k$ is odd. The following lemma is required.
	
	\begin{lemma}\label{auxillary_lemma_k_even}
		Let $k\geq 2$ be even and $l\equiv 2\;(\Mod\; 3)$. Let $n=\frac{k(2l-1)}{3}$.
		\begin{enumerate}
			\item If $k\equiv 2\;(\Mod\;4)$ and $\sigma \in A_{n+5}$ be such that $\sigma$ is a disjoint product of $\frac{n}{2}+1$ many 2-cycles and a 3-cycle, then $\sigma \notin \p(k,l; n+5)$.
			
			\item If $k\equiv 0\;(\Mod\;4)$ and $\sigma \in A_{n+3}$ be such that $\sigma$ is a disjoint product of $\frac{n}{2}$ many 2-cycles and a 3-cycle, then $\sigma \notin \p(k,l; n+3)$.
		\end{enumerate}
	\end{lemma}

	\begin{proof}
		The proof is by induction on $k$. Let $k=2$ and $\sigma \in A_{\frac{4l-2}{3}+5}$ be a disjoint product of $\frac{2l-1}{3}+1$ many 2-cycles and a 3-cycle. Then $2l-m_{\sigma}-n_{\sigma}=2l-\left(\frac{4l-2}{3}+5\right)-\left(\frac{2l-1}{3}+2\right)=-6.$ By \Cref{general_product_condition}, we conclude that $\sigma \notin \p(2,l;\frac{4l-2}{3}+5)$. 
		
		Let $k=4$ and $\sigma \in A_{\frac{8l-4}{3}+3}$ be a disjoint product of $\frac{2(2l-1)}{3}$ many 2-cycles and a 3-cycle. Then
		$4l-m_{\sigma}-n_{\sigma}=4l-\left(\frac{8l-4}{3}+3\right)-\left(\frac{2(2l-1)}{3}+1\right)=-2.$ By \Cref{general_product_condition}, we conclude that $\sigma \notin \p(4,l;\frac{8l-1}{3}+3)$. Thus the lemma holds when $k=2,4$.
		Let us now assume that the result holds for every $k'\leq k-2$. For any $k\geq 6$, we note the following:
		
		\begin{enumerate}[label=(\alph*)]
			\item When $\sigma$ is as in (1), then $kl-m_{\sigma}-n_{\sigma}=kl-\frac{k(2l-1)}{3}-5-\frac{k(2l-1)}{6}-2=\frac{k-14}{2}<k-4.$
			\item When $\sigma$ is as in (2), then $kl-m_{\sigma}-n_{\sigma}=kl-\frac{k(2l-1)}{3}-3-\frac{k(2l-1)}{6}-1=\frac{k-8}{2}<k-4.$
		\end{enumerate}
		By \Cref{indecomposable}, it is enough to show that in each case $\sigma$ is even-$k$-indecomposable. 
		
		\medskip
		
		\textbf{(1)} On the contrary, assume that $\sigma$ is $(k',k-k')$ decomposable for some even $k'$ such that $2\leq k'\leq k-2$. Thus there exists two disjoint permutations $\rho,\tau \in A_n$ such that $\sigma=\rho\tau$ and $\rho\in \p(k',l; n+5)$ and $\tau \in \p(k-k',l; n+5)$. WLOG we may assume that $k'\equiv 0\;(\Mod\;4), k-k'\equiv 2\;(\Mod\; 4)$.
		
		First we assume that $\rho$ is a disjoint product of 2-cycles. If $m_{\rho}\geq \frac{k'(2l-1)}{3}+4$, then by \Cref{UB_for_odd_l}(4) and \Cref{adding_even_many_two_cycles} (since (d) in the proof of \Cref{UB_for_odd_l} holds), we conclude that $\rho \notin \p(k',l;n+5)$, a contradiction. So we assume that $m_{\rho}<\frac{k'(2l-1)}{3}+4$. This yields $m_{\rho}\leq \frac{k'(2l-1)}{3}.$ Then $m_{\tau}=m_{\sigma}-m_{\rho}\geq \frac{(k-k')(2l-1)}{3}+5.$ By induction hypothesis (1) and by \Cref{adding_even_many_two_cycles} (since (a) holds), we conclude that $\tau \notin \p(k-k',l;n+5)$, once again a contradiction.
		
		Now assume that $\rho$ is a disjoint product of  2-cycles and a 3-cycle. If $m_{\rho}\geq \frac{k'(2l-1)}{3}+3$, then by induction hypothesis (2) and \Cref{adding_even_many_two_cycles} (since (b) holds), we conclude that $\rho \notin \p(k',l;n+5)$, a contradiction. So we assume that $m_{\rho}<\frac{k'(2l-1)}{3}+3$. This yields $m_{\rho}\leq \frac{k'(2l-1)}{3}-1.$ Then $m_{\tau}=m_{\sigma}-m_{\rho}\geq \frac{(k-k')(2l-1)}{3}+6.$ By \Cref{UB_for_odd_l}(1) and by \Cref{adding_even_many_two_cycles} (since (a) holds in the proof of \Cref{UB_for_odd_l}), we conclude that $\tau \notin \p(k-k',l;n+5)$, once again a contradiction.   Thus $\sigma$ is even-$k$-indecomposable in this case.
		
		\vspace{3 mm}
		
		\textbf{(2)} On the contrary let us assume that $\sigma$ is $(k',k-k')$ decomposable for some even $k'$ such that $2\leq k'\leq k-2$. Thus there exists two disjoint permutations $\rho,\tau \in A_n$ such that $\sigma=\rho\tau$ and $\rho\in \p(k',l; n+3)$ and $\tau \in \p(k-k',l; n+3)$. There are two possibilities (i) $k'\equiv 0\;(\Mod\;4), k-k'\equiv 0\;(\Mod\; 4)$, (ii) $k'\equiv 2\;(\Mod\;4), k-k'\equiv 2\;(\Mod\; 4)$.
		
		Consider case (i). WLOG we may assume that $\rho$ is a disjoint product of 2-cycles. If $m_{\rho}\geq \frac{k'(2l-1)}{3}+4$, then by \Cref{UB_for_odd_l}(4) and \Cref{adding_even_many_two_cycles} (since (d) in the proof of \Cref{UB_for_odd_l} holds), we conclude that $\rho \notin \p(k',l;n+3)$, a contradiction. So we assume that $m_{\rho}<\frac{k'(2l-1)}{3}+4$. This yields $m_{\rho}\leq \frac{k'(2l-1)}{3}.$ Then $m_{\tau}=m_{\sigma}-m_{\rho}\geq \frac{(k-k')(2l-1)}{3}+3.$ By induction hypothesis (2) and by \Cref{adding_even_many_two_cycles} (since (b) holds), we conclude that $\tau \notin \p(k-k',l;n+3)$, once again a contradiction.
		
		Now consider case (ii). Once again, WLOG we may assume that $\rho$ is a disjoint product of 2-cycles. If $m_{\rho}\geq \frac{k'(2l-1)}{3}+2$, then by \Cref{UB_for_odd_l}(1) and \Cref{adding_even_many_two_cycles} (since (a) in the proof of \Cref{UB_for_odd_l} holds), we conclude that $\rho \notin \p(k',l;n+3)$, a contradiction. So we assume that $m_{\rho}<\frac{k'(2l-1)}{3}+2$. This yields $m_{\rho}\leq \frac{k'(2l-1)}{3}-2.$ Then $m_{\tau}=m_{\sigma}-m_{\rho}\geq \frac{(k-k')(2l-1)}{3}+5.$ By induction hypothesis (1) and by \Cref{adding_even_many_two_cycles} (since (a) holds), we conclude that $\tau \notin \p(k-k',l;n+3)$, once again a contradiction. This proves that $\sigma$ is even-$k$-indecomposable. Thus $\sigma$ is even-$k$-indecomposable in this case as well and the proof is complete.
	\end{proof}

	\begin{proposition}\label{UB_for_l_odd_k_odd_l=2(mod3)}
		Let $l>3$ be odd, $l\equiv2\;(\Mod\;3)$, and $k\geq 3$ be odd. Further let $n=\frac{k(2l-1)}{3}$. Then the following statements hold:
		\begin{enumerate}
			\item Let $k\equiv 3\;(\Mod\;4)$ and let $\sigma \in A_{n+2}$ be a disjoint product of $\frac{n-1}{2}$ many $2$-cycles and one $3$-cycle. Then $\sigma \notin \mathcal{P}(k,l;n+2)$.
			
			\item Let $k\equiv 3\;(\Mod\;4)$ and let $\sigma \in A_{n+3}$ be a disjoint product of $\frac{n+3}{2}$ many $2$-cycles. Then $\sigma \notin \mathcal{P}(k,l;n+3)$.
			
			\item Let $k\equiv 1\;(\Mod\;4)$ and let $\sigma \in A_{n+1}$ be a disjoint product of $\frac{n+1}{2}$ many $2$-cycles. Then $\sigma \notin \mathcal{P}(k,l;n+1)$.
			
			\item Let $k\equiv 1\;(\Mod\;4)$ and let $\sigma \in A_{n+4}$ be a disjoint product of $\frac{n+1}{2}$ many $2$-cycles and one $3$-cycle. Then $\sigma \notin \mathcal{P}(k,l;n+4)$.
		\end{enumerate}
	In particular, $n(k,l)\leq \frac{k(2l-1)}{3}+1$ when $k\equiv 3\;(\Mod\;4)$ and $n(k,l)\leq \frac{k(2l-1)}{3}$ when $k\equiv 1\;(\Mod\;4)$.
	\end{proposition}	
	
	\begin{proof}
		The proof is by induction on $k$. For $k=3$, we first consider the statement (1). 
		Let $n=\frac{k(2l-1)}{3}=2l-1$ and let $\sigma \in A_{2l+1}$ be a disjoint product of $l-1$ many $2$-cycles and one $3$-cycle. Then
		we have $3l-m_{\sigma}-n_{\sigma}=3l-(2l+1)-l=-1 <0$. Thus $\sigma \notin \mathcal{P}(3,l;2l+1)$ by \Cref{general_product_condition}.  
		For statement (2), let $\sigma \in A_{2l+2}$ be a disjoint product of $l+1$ many $2$-cycles. Then $3l-m_{\sigma}-n_{\sigma}=3l-2l-2-l-1=-3 <0$. Thus $\sigma \notin \mathcal{P}(3,l;2l+2)$ once again by \Cref{general_product_condition}.

		For $k=5$, we first consider the statement $(3)$. We have $n=\frac{5(2l-1)}{3}=\frac{10l-5}{3}$ and let $\sigma \in A_{\frac{10l-2}{3}}$ 
		be a disjoint product of $\frac{10l-2}{6}$ many $2$-cycles. Then $5l-m_{\sigma}-n_{\sigma}=5l-\frac{10l-2}{3}-\frac{10l-2}{6}=1<3$.
		By \Cref{indecomposable}, in order to show that $\sigma \notin \mathcal{P}(5,l;\frac{10l-2}{3})$, it is enough to show that $\sigma$ is $(2,3)$-indecomposable. Suppose that there exist $\rho,\tau$ disjoint permutations in $A_{\frac{10l-2}{3}}$ such that $\sigma =\rho \tau$ and $\rho\in \p(2,l;m_{\rho})$, $\tau \in \p(3,l;m_{\tau})$. If $m_{\rho}\geq \frac{4l+4}{3}$, then $n_{\rho}\geq \frac{2(l+1)}{3}$ and $2l-m_{\rho}-n_{\rho}\leq -2<0$, which implies $\rho \notin \p(2,l;m_{\tau})$ by \Cref{general_product_condition}, a contradiction. Note that $m_{\rho}\equiv 0\; (\Mod\;4)$. So let $m_{\rho}\leq \frac{4l-2}{3}$. Then $m_{\tau}=m_{\sigma}-m_{\rho}\geq 2l+2$ and $n_{\tau}\geq l+1$. 
		Thus $3l-m_{\tau}-n_{\tau}\leq -3<0$, which implies that $\tau\notin \p(3,l;m_{\tau})$ by \Cref{general_product_condition}, a contradiction. Hence $\sigma$ is $(2,3)$-indecomposable as desired.

		For statement (4), let $\sigma \in A_{\frac{10l+7}{3}}$ be a disjoint product of $\frac{10l-2}{6}$ many $2$-cycles and one $3$-cycle. Note that $5l-m_{\sigma}-n_{\sigma}=5l-\frac{10l+7}{3}-\frac{10l+4}{6}=-3 <0$. Thus $\sigma \notin \mathcal{P}(5,l;\frac{10l+7}{3})$ by \Cref{general_product_condition}.
		
		\medskip
		
		Assume that the result holds for all odd $3\leq k'<k$. We want to prove the result for $k$. The followings hold for every $k$ such that $k\geq 3$.
		
		\begin{enumerate}[label=(\alph*)]
			\item If $\sigma$ is as in (1), then $kl-m_{\sigma}-n_{\sigma}= kl-\frac{k(2l-1)}{3}-2-\frac{k(2l-1)-3}{6}-1< k-2$.
			\item If $\sigma$ is as in (2), then $kl-m_{\sigma}-n_{\sigma}= kl-\frac{k(2l-1)}{3}-3-\frac{k(2l-1)+9}{6}< k-2$.
			\item If $\sigma$ is as in (3), then $kl-m_{\sigma}-n_{\sigma}= kl-\frac{k(2l-1)}{3}-1-\frac{k(2l-1)+3}{6}< k-2$.
			\item If $\sigma$ is as in (4), then $kl-m_{\sigma}-n_{\sigma}= kl-\frac{k(2l-1)}{3}-4-\frac{k(2l-1)+3}{6}-1< k-2$.
		\end{enumerate}
		
		\noindent By \Cref{indecomposable}, it is enough to show that $\sigma$ in each case is $k$-indecomposable which we now establish.
		
		\medskip
		
		\textbf{(1)} On the contrary, suppose that there exist two disjoint permutations $\rho, \tau \in A_n$ such that $\sigma=\rho \tau$ and $\rho \in \p(k',l;m_{\rho})$ and $\tau \in \p(k-k',l;m_{\tau})$ for some $k'$, $2\leq k' \leq k-2$. WLOG we may assume that $k'$ is even. 
		
		\medskip
		
		\textbf{\underline{Case A}:} Suppose that $k'\equiv 0\;(\Mod\;4)$ and $\rho$ is a disjoint product of 2-cycles only. Note that $k-k'\equiv 3\;(\Mod\; 4)$. If $m_{\rho}\geq \frac{k'(2l-1)}{3}+4$ then by \Cref{UB_for_odd_l}(4) and \Cref{adding_even_many_two_cycles} (since (d) in the proof of \Cref{UB_for_odd_l} holds), $\rho\notin \p(k',l;m_{\rho})$, which is a contradiction. Let $m_{\rho}\leq \frac{k'(2l-1)}{3}$. Then $m_{\tau}=m_{\sigma}-m_{\rho}\geq \frac{(k-k')(2l-1)}{3}+2$. Thus by induction hypothesis $(1)$ and \Cref{adding_even_many_two_cycles} (since (a) holds), $\tau \notin \p(k-k',l;m_{\tau})$, which is a contradiction. 
		
		\medskip
		
		\textbf{\underline{Case B}:} Suppose that $k'\equiv 0\;(\Mod\;4)$ and $\rho$ is a disjoint product of one $3$-cycle and $n_{\rho}-1$ many $2$-cycles. If $m_{\rho}\geq \frac{k'(2l-1)}{3}+3$ then by \Cref{auxillary_lemma_k_even}(2) and \Cref{adding_even_many_two_cycles} (since (b) in the proof of \Cref{auxillary_lemma_k_even} holds), $\rho\notin \p(k',l;m_{\rho})$. Let $m_{\rho}\leq \frac{k'(2l-1)}{3}-1$. Then $m_{\tau}=m_{\sigma}-m_{\rho}\geq \frac{(k-k')(2l-1)}{3}+3$ and hence by induction hypothesis $(2)$ and \Cref{adding_even_many_two_cycles} (since (b) holds), $\tau\notin \p(k-k',l;m_{\tau})$, which is a contradiction. 
		
		\medskip
		
		\textbf{\underline{Case C}:} Suppose that $k'\equiv 2\;(\Mod\;4)$ and $\rho$ is a disjoint product of $2$-cycles only. Note that $k-k'\equiv 1\;(\Mod\;4)$. 
		If $m_{\rho}\geq \frac{k'(2l-1)}{3}+2$ then by \Cref{UB_for_odd_l}(1) and \Cref{adding_even_many_two_cycles} (since (a) in the proof of \Cref{UB_for_odd_l} holds), $\rho\notin \p(k',l;m_{\rho})$. Suppose that $m_{\rho}\leq \frac{k'(2l-1)}{3}-2$. Then $m_{\tau}\geq \frac{(k-k')(2l-1)}{3}+4$. Thus, by induction hypothesis $(4)$ and \Cref{adding_even_many_two_cycles} (since (d) holds), $\tau\notin \p(k-k',l;m_{\tau})$, which is once again a contradiction.
		
		\medskip
		
		\textbf{\underline{Case D}:} Suppose that $k'\equiv 2\;(\Mod\;4)$ and $\rho$ is a disjoint product of one $3$-cycle and $n_{\rho}-1$ many $2$-cycles only. 
		If $m_{\rho}\geq \frac{k'(2l-1)}{3}+5$ then by \Cref{auxillary_lemma_k_even}(1) and \Cref{adding_even_many_two_cycles} (since (a) in the proof of \Cref{auxillary_lemma_k_even} holds), $\rho\notin \p(k',l;m_{\rho})$. Suppose that $m_{\rho}\leq \frac{k'(2l-1)}{3}+1$. Then $m_{\tau}=m_{\sigma}-m_{\rho}\geq \frac{(k-k')(2l-1)}{3}+1$. Thus, by induction hypothesis (3) and \Cref{adding_even_many_two_cycles} (since (c) holds), $\tau\notin \p(k-k',l;m_{\tau})$.
		We get our final contradiction and conclude that $\sigma$ is $k$-indecomposable.
		
		\medskip
		
		\textbf{(2)} On the contrary, suppose that there exist two disjoint permutations $\rho, \tau \in A_n$ such that $\sigma=\rho \tau$ and $\rho \in P(k',l;m_{\rho})$ and $\tau \in P(k-k',l;m_{\tau})$ for some $k'$, $2\leq k' \leq k-2$. WLOG we may assume that $k'$ is even.  
		
		\medskip
		
		\textbf{\underline{Case A}:} Suppose that $k'\equiv 0\;(\Mod\;4)$. If $m_{\rho}\geq \frac{k'(2l-1)}{3}+4$ then by \Cref{UB_for_odd_l}(4) and \Cref{adding_even_many_two_cycles} (since (d) in the proof of \Cref{UB_for_odd_l} holds), $\rho \notin \p(k',l;m_{\rho})$, a contradiction. Assume that $m_{\rho}\leq \frac{k'(2l-1)}{3}$. Then $m_{\tau}=m_{\sigma}-m_{\rho}\geq m_{\rho}\geq \frac{(k-k')(2l-1)}{3}+3$. By induction hypothesis (2) (as $k-k'\equiv 3\;(\Mod\;4)$) and \Cref{adding_even_many_two_cycles} (since (b) holds), $\tau\notin \p(k-k',l:m_{\tau})$, which is once again a contradiction.
		
		\medskip
		
		\textbf{\underline{Case B}:} Suppose that $k'\equiv 2\;(\Mod\;4)$. If $m_{\rho}\geq \frac{k'(2l-1)}{3}+2$ then by \Cref{UB_for_odd_l}(1) and \Cref{adding_even_many_two_cycles} (since (a) in the proof of \Cref{UB_for_odd_l} holds), $\rho \notin \p(k',l;m_{\rho})$, a contradiction. Assume that $m_{\rho}\leq \frac{k'(2l-1)}{3}-2$. Then $m_{\tau}=m_{\sigma}-m_{\rho}\geq m_{\rho}\geq \frac{(k-k')(2l-1)}{3}+5$ and by induction hypothesis (3) and \Cref{adding_even_many_two_cycles} (since (c) holds), $\tau\notin \p(k-k',l:m_{\tau}+3)$. We get our final contradiction and conclude that $\sigma$ is $k$-indecomposable.
		
		\medskip
		
		\textbf{(3)} On the contrary, suppose that there exist two disjoint permutations $\rho, \tau \in A_n$ such that $\sigma=\rho \tau$ and $\rho \in P(k',l;m_{\rho})$ and $\tau \in P(k-k',l;m_{\tau})$ for some $k'$, $2\leq k' \leq k-2$. WLOG we may assume that $k'$ is even.  
		
		\medskip
		
		\textbf{\underline{Case A}:} Suppose that $k' \equiv 0\;(\Mod\;4)$. Then 
		$\frac{k'(2l-1)}{3}\equiv 0\;(\Mod\;4)$. If $m_{\rho}\geq \frac{k'(2l-1)}{3}+4$ then by \Cref{UB_for_odd_l}(4) and \Cref{adding_even_many_two_cycles} (since (d) in the proof of \Cref{UB_for_odd_l} holds), $\rho\notin \p(k',l;m_{\rho})$, which is a contradiction. Assume that $m_{\rho}\leq \frac{k'(2l-1)}{3}$. Then $m_{\tau}=m_{\sigma}-m_{\rho}\geq \frac{(k-k')(2l-1)}{3}+1$ and thus by induction hypothesis $(3)$ and \Cref{adding_even_many_two_cycles} (since (c) holds), $\tau \notin \p(k-k',l;m_{\tau})$, which is once again a contradiction.
		
		\medskip
		
		\textbf{\underline{Case B}:} Suppose that $k' \equiv 2\;(\Mod\;4)$. Then 
		$\frac{k'(2l-1)}{3}\equiv 2\;(\Mod\;4)$. If $m_{\rho}\geq \frac{k'(2l-1)}{3}+2$ then by \Cref{UB_for_odd_l}(1) and \Cref{adding_even_many_two_cycles} (since (a) in the proof of \Cref{UB_for_odd_l} holds), $\rho\notin \p(k',l;m_{\rho})$, which is a contradiction. Assume that $m_{\rho}\leq \frac{k'(2l-1)}{3}-2$. Then $m_{\tau}=m_{\sigma}-m_{\rho}\geq \frac{(k-k')(2l-1)}{3}+3$. By induction hypothesis (2) and \Cref{adding_even_many_two_cycles} (since (b) holds), $\tau \notin \p(k-k',l;m_{\tau})$.  This is also a contradiction. We conclude that $\sigma$ is $k$-indecomposable.
		
		\medskip
		
		\textbf{(4)} On the contrary, suppose that there exist two disjoint permutations $\rho, \tau \in A_n$ such that $\sigma=\rho \tau$ and $\rho \in P(k',l;m_{\rho})$ and $\tau \in P(k-k',l;m_{\tau})$ for some $k'$, $2\leq k' \leq k-2$. WLOG we may assume that $k'$ is even.  
		
		\medskip
		
		\textbf{\underline{Case A}:} Assume that $k'\equiv 0\;(\Mod \;4)$ and $\rho$ is a disjoint product of $2$-cycles only. If $m_{\rho}\geq \frac{k'(2l-1)}{3}+4$ then by
		\Cref{UB_for_odd_l}(4) and \Cref{adding_even_many_two_cycles} (since (d) in the proof of \Cref{UB_for_odd_l} holds), $\rho\notin \p(k',l;m_{\rho})$, which is a contradiction. Assume that $m_{\rho}\leq \frac{k'(2l-1)}{3}$. Then $m_{\tau}=m_{\sigma}-m_{\rho}\geq \frac{(k-k')(2l-1)}{3}+4$, and thus by induction hypothesis (4) as $k-k'\equiv 1\;(\Mod\;4)$  and \Cref{adding_even_many_two_cycles} (since (d) holds), $\tau\notin \p(k-k',l;m_{\tau})$, which is a contradiction. 
		
		\medskip
		
		\textbf{\underline{Case B}:}  Assume that $k'\equiv 0\;(\Mod \;4)$ and $\rho$ is a disjoint product of one $3$-cycle and $n_{\rho}-1$ many $2$-cycles only. 
		If $m_{\rho}\geq \frac{k'(2l-1)}{3}+3$ then by \Cref{auxillary_lemma_k_even}(2) and \Cref{adding_even_many_two_cycles} (since (b) in the proof of \Cref{auxillary_lemma_k_even} holds), $\rho\notin \p(k',l;m_{\rho})$, which is a contradiction. Assume that $m_{\rho}\leq \frac{k'(2l-1)}{3}-1$. Then $m_{\tau}=m_{\sigma}-m_{\rho}\geq \frac{(k-k')(2l-1)}{3}+5$, and thus by induction hypothesis $(3)$ as $k-k'\equiv 1\;(\Mod\; 4)$ and \Cref{adding_even_many_two_cycles} (since (c) holds), $\tau\notin \p(k-k',l;m_{\tau})$, which is a contradiction.
		
		\medskip
		
		\textbf{\underline{Case C}:} Assume that $k'\equiv 2\;(\Mod \;4)$ and $\rho$ is a disjoint product of $2$-cycles only. If $m_{\rho}\geq \frac{k'(2l-1)}{3}+2$ then by \Cref{UB_for_odd_l}(1) and \Cref{adding_even_many_two_cycles} (since (a) in the proof of \Cref{UB_for_odd_l} holds), $\rho\notin \p(k',l;m_{\rho})$, which is a contradiction. Assume that $m_{\rho}\leq \frac{k'(2l-1)}{3}-2$. Then $m_{\tau}=m_{\sigma}-m_{\rho}\geq \frac{(k-k')(2l-1)}{3}+6$, and thus by induction hypothesis (1) as $k-k'\equiv 3\;(\Mod\; 4)$ and \Cref{adding_even_many_two_cycles} (since (a) holds), $\tau\notin \p(k-k',l;m_{\tau})$, which is a contradiction. 
		
		\medskip
		
		\textbf{\underline{Case D}:} Assume that $k'\equiv 2\;(\Mod \;4)$ and $\rho$ is a disjoint product of one $3$-cycle and $n_{\rho}-1$ many $2$-cycles. 
		If $m_{\rho}\geq \frac{k'(2l-1)}{3}+5$ then by \Cref{auxillary_lemma_k_even}(1) and \Cref{adding_even_many_two_cycles}, $\rho\notin \p(k',l;m_{\rho})$ (since (a) in the proof of \Cref{auxillary_lemma_k_even} holds), which is a contradiction. Assume that $m_{\rho}\leq \frac{k'(2l-1)}{3}+1$. Then $m_{\tau}=m_{\sigma}-m_{\rho}\geq \frac{(k-k')(2l-1)}{3}+3$, and thus by induction hypothesis (2) and \Cref{adding_even_many_two_cycles} (since (b) holds), $\tau\notin \p(k-k',l;m_{\tau})$. This is our final contradiction and we conclude that $\sigma$ is $k$-indecomposable. The proof is now complete.
	\end{proof}
	
	\begin{proof}[\textbf{Proof of \Cref{maintheorem}}]
		We have proved the theorem for $l\equiv 1\;(\Mod\; 3)$ in \Cref{Proof_of_main_theorem_l=1(mod 3)}. For $k$ even and $l\equiv 2\;(\Mod\; 3)$ it is proved in \Cref{Proof_of_main_theorem_k_even_l=2(mod 3)}. For $k$ odd, $l$ odd, and $l\equiv 2\;(\Mod\; 3)$, the proof follows from \Cref{lower_bound_for_k_odd_l_odd_l=2(mod3)} and \Cref{UB_for_l_odd_k_odd_l=2(mod3)}.
		
	\end{proof}

	\subsection*{Acknowledgment} The first named author acknowledges the support through Prime Minister's Research Fellowship from the Ministry of Education, Government of India (PMRF ID: 0601097). The second and third named authors would like to acknowledge the support of IISER Mohali institute post-doctoral fellowship during this work. We also thank Prof. Amit Kulshrestha for his interest in this work. The third named author thanks his post-doctoral mentor Prof. Chetan Balwe for his support.
	\bibliographystyle{amsalpha}
	\bibliography{References}

\newcommand{\etalchar}[1]{$^{#1}$}
\providecommand{\bysame}{\leavevmode\hbox to3em{\hrulefill}\thinspace}
\providecommand{\MR}{\relax\ifhmode\unskip\space\fi MR }
\providecommand{\MRhref}[2]{%
  \href{http://www.ams.org/mathscinet-getitem?mr=#1}{#2}
}
\providecommand{\href}[2]{#2}
\begin{thebibliography}{GLO{\etalchar{+}}18}

\bibitem[ASH85]{ash}
Z.~Arad, J.~Stavi, and M.~Herzog, \emph{Powers and products of conjugacy
  classes in groups}, Products of conjugacy classes in groups, Lecture Notes in
  Math., vol. 1112, Springer, Berlin, 1985, pp.~6--51. \MR{783068}

\bibitem[Ber72]{be}
Edward Bertram, \emph{Even permutations as a product of two conjugate cycles},
  J. Combinatorial Theory Ser. A \textbf{12} (1972), 368--380. \MR{297853}

\bibitem[BFM20]{bfm}
Antonio Beltr\'{a}n, Mar\'{\i}a~Jos\'{e} Felipe, and Carmen Melchor, \emph{Some
  problems about products of conjugacy classes in finite groups}, Int. J. Group
  Theory \textbf{9} (2020), no.~1, 59--68. \MR{4076427}

\bibitem[BH01]{bh}
Edward Bertram and Marcel Herzog, \emph{Powers of cycle-classes in symmetric
  groups}, J. Combin. Theory Ser. A \textbf{94} (2001), no.~1, 87--99.
  \MR{1816248}

\bibitem[Ch65]{xu}
Xu~Cheng-hao, \emph{The commutators of the alternating group}, Sci. Sinica
  \textbf{14} (1965), 339--342. \MR{183763}

\bibitem[EG98]{eg}
Erich~W. Ellers and Nikolai Gordeev, \emph{On the conjectures of {J}.
  {T}hompson and {O}. {O}re}, Trans. Amer. Math. Soc. \textbf{350} (1998),
  no.~9, 3657--3671. \MR{1422600}

\bibitem[FLS75]{fls}
Walter Feit, Roger Lyndon, and Leonard~L. Scott, \emph{A remark about
  permutations}, J. Combinatorial Theory Ser. A \textbf{18} (1975), 234--335.
  \MR{372000}

\bibitem[GLO{\etalchar{+}}18]{GLOST}
Robert~M. Guralnick, Martin~W. Liebeck, E.~A. O'Brien, Aner Shalev, and
  Pham~Huu Tiep, \emph{Surjective word maps and {B}urnside's {$p^aq^b$}
  theorem}, Invent. Math. \textbf{213} (2018), no.~2, 589--695. \MR{3827208}

\bibitem[GM12]{gm}
Robert Guralnick and Gunter Malle, \emph{Products of conjugacy classes and
  fixed point spaces}, J. Amer. Math. Soc. \textbf{25} (2012), no.~1, 77--121.
  \MR{2833479}

\bibitem[HKL04]{hkl}
Marcel Herzog, Gil Kaplan, and Arieh Lev, \emph{Representation of permutations
  as products of two cycles}, Discrete Mathematics \textbf{285} (2004), no.~1,
  323--327.

\bibitem[HKL08]{hgl}
Marcel Herzog, Gil Kaplan, and Arieh Lev, \emph{Covering the alternating groups
  by products of cycle classes}, J. Combin. Theory Ser. A \textbf{115} (2008),
  no.~7, 1235--1245. \MR{2450339}

\bibitem[KKM22]{kkm}
Harish Kishnani, Rijubrata Kundu, and Sumit~Chandra Mishra, \emph{Alternating
  groups as products of cycle classes}, https://arxiv.org/abs/2207.03165
  (2022).

\bibitem[Mal14]{ma}
Gunter Malle, \emph{The proof of {O}re's conjecture (after {E}llers-{G}ordeev
  and {L}iebeck-{O}'{B}rien-{S}halev-{T}iep)}, Ast\'{e}risque (2014), no.~361,
  Exp. No. 1069, ix, 325--348. \MR{3289286}

\bibitem[Ree71]{re}
Rimhak Ree, \emph{A theorem on permutations}, J. Combinatorial Theory Ser. A
  \textbf{10} (1971), 174--175. \MR{269519}

\bibitem[Sha13]{sh}
Aner Shalev, \emph{Some results and problems in the theory of word maps},
  Erd\"{o}s centennial, Bolyai Soc. Math. Stud., vol.~25, J\'{a}nos Bolyai
  Math. Soc., Budapest, 2013, pp.~611--649. \MR{3203613}

\end{thebibliography}
\end{document}